\documentclass[11pt,a4paper]{article}

\usepackage[utf8]{inputenc}
\usepackage[T1]{fontenc}
\usepackage[english]{babel}
\usepackage{amsmath,amsthm,amssymb,amsfonts}
\usepackage{mathtools}

\usepackage[margin=2.5cm]{geometry}
\usepackage{setspace}
\theoremstyle{plain}
\newtheorem{theorem}{Theorem}[section]
\newtheorem{lema}[theorem]{Lemma}
\newtheorem{proposition}[theorem]{Proposition}
\newtheorem{coro}[theorem]{Corollary}

\theoremstyle{definition}
\newtheorem{definition}[theorem]{Definition}
\newtheorem{example}[theorem]{Example}
\newtheorem{remark}[theorem]{Remark}
\newtheorem{asumption}{Assumption}

\usepackage{xcolor}
\usepackage[colorlinks=true,
            linkcolor=blue,
            citecolor=blue,
            urlcolor=blue,
            breaklinks=true]{hyperref}

\title{\textbf{Process-Based Lagrange Multipliers for Nonconvex Set-Valued Optimization}}

\author{
Fernando García-Castaño\thanks{Department of Mathematics, University of Alicante, Carretera San Vicente del Raspeig, s/n, 03690 San Vicente del Raspeig, Alicante, Spain. Email: \texttt{fernando.gc@ua.es}} 
\and 
Miguel Ángel Melguizo-Padial\thanks{Department of Mathematics, University of Alicante. Email: \texttt{ma.mp@ua.es}}
}

\date{\today} 

\begin{document}

\maketitle

\begin{abstract}
We develop a Lagrange multiplier theory for nonconvex set-valued optimization problems under Lipschitz-type regularity conditions. In contrast to classical approaches that rely on continuous linear functionals, we introduce closed convex processes---set-valued mappings whose graphs form closed convex cones---as generalized Lagrange multipliers. This geometric framework extends fundamental separation principles from convex analysis to settings where neither convexity nor differentiability is assumed.

Our main results establish the existence of multiplier processes under verifiable structural conditions: Lipschitz continuity at a reference point, existence of a bounded base for the ordering cone, and a nondegeneracy assumption ensuring proper isolation of optimal values. We prove that every such process preserves global optimality in the sense that nondominated (respectively, minimal) points of the primal problem remain nondominated (respectively, minimal) in the augmented problem obtained by penalizing constraints through the process.

For scalar optimization, we establish a one-to-one correspondence between multiplier processes and lower semicontinuous sublinear functions, yielding exact penalty formulations without additional constraint qualifications. An illustrative example shows that the nonemptiness of the core of the ordering cone, although sufficient, is not a necessary assumption, and that it holds in common Banach spaces.

As an application, we show how the abstract framework naturally applies to set-valued vector equilibrium problems, where the equilibrium condition is embedded as a constraint and the multiplier process encodes admissible directions of comparison. An explicit example illustrates the construction in a concrete infinite-dimensional setting.

Our results provide a rigorous geometric foundation for global optimality conditions in nonconvex optimization, complementing classical theory and offering new structural insights that may inform future algorithmic developments.
\end{abstract}



\section{Introduction}\label{sec:introduction}
The classical theory of Lagrange multipliers is traditionally divided into two main frameworks: a global theory, which applies to convex optimization problems with inequality constraints, and a local theory, which applies to differentiable problems with equality and/or inequality constraints. Despite the different mathematical settings, the underlying principles of both theories are essentially the same. In each case, the existence of a continuous linear functional is required—either to support a feasible set or to be tangent to appropriate level sets and constraint boundaries; see \cite{Luenberger1969}.

When convexity or differentiability assumptions fail, one might expect that the classical theory could still be carried out provided that appropriate continuous linear functionals were available. However, the classical Lagrange multiplier framework breaks down precisely in those situations where such functionals cannot be found or do not exist. This limitation raises a natural question: Can we extend the geometric principles underlying Lagrange multiplier theory to nonconvex and nondifferentiable settings by replacing linear functionals with more general geometric objects?

In this work, we develop a theory of Lagrange multipliers in which continuous linear functionals are replaced by closed convex processes. This new approach allows us to handle problems that lack convexity or differentiability, and more importantly, situations where continuous linear functionals cannot be determined---thus complementing and extending the classical theory. The key insight is that when approximating the graph of a function through a process (i.e., via a cone), fewer technical conditions are needed compared to approximation using a hyperplane. Nevertheless, although the theory with processes imposes fewer structural restrictions, it is complete in the sense that Lagrange multipliers always exist under verifiable geometric conditions, even in nonconvex and nondifferentiable problems.


\subsection{Motivation and main contributions}

Classical Lagrange theory for convex or differentiable problems requires a hyperplane to support the graph of the objective function. By contrast, our framework employs closed convex processes---set-valued mappings whose graphs form closed convex cones---as generalized multipliers. This geometric perspective offers several key advantages. First, the process-based approach extends applicability to optimization problems involving set-valued mappings under Lipschitz-type regularity conditions, without requiring convexity or differentiability of the objective or constraint functions. Second, we extend fundamental separation principles from convex analysis to nonconvex regimes, providing a rigorous theoretical foundation that deepens our understanding of global optimality structures, paralleling the well-established role of optimality conditions in convex optimization. Third, we provide concrete, verifiable conditions---involving Lipschitz continuity at a point, bounded bases for ordering cones, and nondegeneracy assumptions---that ensure the existence of Lagrange multiplier processes and are satisfied in many classical ordered Banach spaces. Finally, we demonstrate how our abstract framework applies naturally to set-valued vector equilibrium problems, where the equilibrium condition is embedded into a parametric optimization framework and the resulting multiplier encodes admissible directions of comparison.

This geometric perspective reveals new structural insights into nonconvex problems and establishes a mathematical basis that may help guide the development of future algorithmic approaches, much as classical convex optimality theory (see, e.g., \cite{Rockafellar1970,KTZ2015,GoeRiaTamZal2023}) has influenced the design of computational methods over the past decades (see, e.g., \cite{NesterovNemirovski1994,BoydVandenberghe2004,ParikhBoyd2014,BauschkeCombettes2017,BoydADMM2011}).

\subsection{Structure of the paper}
The remainder of this paper is organized as follows. Section~\ref{sec:preliminaries} introduces the necessary notation and preliminary results on cones, processes, and set-valued mappings in normed spaces. In Section~\ref{sec:general_framework}, we present the general Lagrange multiplier framework for parametric set-valued optimization problems. We define the notion of a Lagrange multiplier as a closed convex process and state the main structural assumptions (Assumption~\ref{hipotesisGeneral}) under which our theory is developed. The central result, Theorem~\ref{ThLagMult11}, establishes that every process satisfying these geometric conditions acts as a Lagrange multiplier, preserving minimality or nondominance in the augmented problem. We specialize the theory to the scalar case ($Y = \mathbb{R}$) and establish a one-to-one correspondence between multiplier processes and lower semicontinuous sublinear functions (Proposition~\ref{proposition:biyeccion_Gamma_Gamma'} and Theorem~\ref{ThLagMult11_real}). In addition, a single-valued version of our theory is provided (Theorem~\ref{teorema_LM_caso_escalar_univalorado}). Two illustrative examples demonstrate the role and scope of the core conditions of the corresponding ordering cones.

Section~\ref{sec:verification_lipschitz} is devoted to verifying the general assumptions under Lipschitz-type regularity and structural conditions on the ordering cone. The main result, Theorem~\ref{teor:TML}, provides verifiable conditions---Lipschitz continuity of the composite mapping at the origin, existence of a bounded base for the ordering cone, and a nondegeneracy condition---that guarantee the nonemptiness of the family of admissible Lagrange multipliers. We develop technical tools involving conic dilations and separation arguments, culminating in Theorem~\ref{teor:TML_real} (scalar case) and Theorem~\ref{teor:TML_real_single_valued} (single-valued case), which prove that the optimal value can be recovered through exact penalization with explicitly constructed multipliers under minimal regularity assumptions.

In Section~\ref{sec:equilibrium_application}, we show how the abstract theory applies to set-valued vector equilibrium problems, a class of problems that has been extensively studied in the literature; see, e.g., the monographs \cite{KTZ2015,GoeRiaTamZal2023} and the references therein. By embedding the equilibrium condition as a constraint, we reformulate vector equilibrium problems within the process-based Lagrange framework and interpret the resulting multiplier process as a geometric object encoding admissible directions of comparison and constraint interaction. An explicit infinite-dimensional example in $C[0,1]$ illustrates the construction and concrete role of the multiplier process in this equilibrium setting, thereby providing a unified geometric interpretation of equilibrium conditions in terms of variational optimality.

\subsection{Relation to prior work}

The use of processes in duality theory has its roots in the work of Rockafellar on conjugate duality and the calculus of set-valued mappings \cite{Rockafellar1970}. The notion of a process as a set-valued mapping with a conical graph was systematically studied by Aubin and Frankowska \cite{aubin1990set}, who developed a comprehensive theory of linear processes and their role in variational analysis. Our contribution extends these ideas by introducing processes as Lagrange multipliers in nonconvex set-valued optimization, thereby providing a genuine generalization of classical multiplier theory that does not rely on convexity or differentiability.

In the context of convex set-valued optimization, the authors previously developed a Lagrange duality framework using pointed processes \cite{GC-Melguizo2022}. The present work removes the convexity assumption entirely and introduces new geometric separation arguments based on conic dilations and Lipschitz regularity at a point, which are essential for handling the nonconvex case.

\section{Preliminaries and Notation}\label{sec:preliminaries}

Let $Y$ and $Z$ be normed spaces with topological duals $Y^*$ and $Z^*$. For $A \subset Y$, we write $\operatorname{cl}(A)$, $\operatorname{int}(A)$, $\operatorname{core}(A)$, $\operatorname{bd}(A)$, $A^c$, $\operatorname{co}(A)$, and $\operatorname{cone}(A)$ for the closure, interior, core, boundary, complement, convex hull, and conic hull of $A$, respectively; parentheses may be omitted when the context is clear. We denote by $\|\cdot\|$ the norm on $Y$, by $\|\cdot\|_*$ the dual norm on $Y^*$, and by $0_Y$ the origin of $Y$. Let $B_Y := \{y \in Y : \|y\| \le 1\}$ and $B_Y^\circ := \{y \in Y : \|y\| < 1\}$ be the closed and open unit balls, respectively, and $S_Y := \{y \in Y : \|y\| = 1\}$ the unit sphere. For $y \in Y$ and $r>0$, we set $B(y,r) := \{y' \in Y : \|y'-y\| \le r\}$ and $B^\circ(y,r) := \{y' \in Y : \|y'-y\| < r\}$. Recall that the core of a set $A \subset Y$ is defined as 
\[
\operatorname{core}(A) := \{\, a \in A : \forall x \in Y,\ \exists \delta > 0 \text{ such that } a+[0,\delta]x \subset A \,\}.
\]
We consider the sum of two subsets in $Y$ in the usual way, adopting the convention $A+\varnothing=\varnothing+A=\varnothing$ for every subset $A \subset Y$. 

A non-empty subset $K\subset Y$ is called a \emph{cone} if $\alpha K \subset K$ for all $\alpha \in \mathbb{R}_+$, where $\mathbb{R}_+$ denotes the set of non-negative real numbers. A cone $K\subset Y$ is said to be \emph{non-trivial} if $\{0_Y\}\subsetneq K \subsetneq Y$. All cones in this manuscript are assumed to be non-trivial unless stated otherwise. A cone $K\subset Y$ is said to be \emph{pointed} if $K \cap (-K)=\{ 0_Y\}$ and \emph{solid} if $\operatorname{int}(K)\neq \varnothing$.

A non-empty convex subset $B$ of a convex cone $K\subset Y$ is said to be a \emph{base} for $K$ if $0_Y \notin \operatorname{cl}(B)$ and for every $y \in K\setminus\{0_Y\}$ there exist unique $\lambda_y > 0$ and $b_y\in B$ such that $y=\lambda_y b_y$.

Given a cone $K\subset Y$, its \emph{dual cone} is defined by 
\[
K^*:=\{f \in Y^*\colon f(y)\geq 0 \ \text{for all } y \in K\},
\]
and the \emph{quasi-relative interior} of $K^*$ by 
\[
K^{\#}:=\{f \in Y^*\colon f(y)> 0 \ \text{for all } y \in K\setminus\{0_Y\}\}.
\]
It is known that a convex cone $C\subset Y$ has a base if and only if $C^{\#}\neq \varnothing$, and the latter implies that $C$ is pointed. In particular, for every $f \in C^{\#}$ and $\lambda>0$, the set $B:=\{x \in C\colon f(x)=\lambda\}$ is a base for $C$ (see, e.g., \cite[Theorem~1.47]{AliprantisTourky2007}).

A convex cone $K$ is said to have a \emph{bounded base} if there exists a base $B$ for $K$ such that $B$ is a bounded subset of $Y$. It is known that $K$ has a bounded base if and only if $0_Y$ is a denting point for $K$ (see \cite{GARCIACASTANO20151178,GARCIACASTANO2019,GARCIACASTANO2021} for further information about dentability and optimization).

A mapping $p:Y\to \mathbb{R}$ is said to be \emph{sublinear} if it is positively homogeneous and subadditive, that is, for all $x,y \in Y$ and $\lambda \ge 0$, 
\[
p(\lambda x) = \lambda p(x), \quad p(x+y) \le p(x) + p(y).
\]

A mapping $f:Y\to\mathbb{R}$ defined on a topological space $Y$ is said to be \emph{lower semicontinuous} at a point $y_0\in Y$ if 
\[
\liminf_{y\to y_0} f(y) \ge f(y_0),
\]
and lower semicontinuous on $Y$ if it is lower semicontinuous at every point of $Y$.

Given a set-valued map $F: Z \rightrightarrows Y$, we identify $F$ with its graph, which is defined by 
\[
\operatorname{Gph}(F):=\{(z,y)\in Z\times Y \colon y\in F(z)\}.
\]
The \emph{domain} of $F$ is defined by $\operatorname{Dom}(F):=\{z\in Z: F(z) \neq \varnothing\}$ and the \emph{image} of $F$ by $\operatorname{Im}(F):=\bigcup_{z \in \operatorname{Dom}(F)} F(z)$. For any set $A\subset \operatorname{Dom}(F)$, we denote the image of $A$ under $F$ by $F(A):=\bigcup_{a \in A}F(a)$.

A set-valued map $\Delta:Z \rightrightarrows Y$ is said to be a \emph{process} if $\operatorname{Graph}(\Delta)$ is a cone. A process $\Delta$ is said to be \emph{convex} (resp.\ \emph{closed}, \emph{pointed}) if $\operatorname{Graph}(\Delta)$ is convex (resp.\ closed, pointed), \emph{proper} if $\operatorname{Dom}(\Delta)=Z$, \emph{real} if $Y=\mathbb{R}$, and \emph{constant} if there exists a cone $K\subset Y$ such that $\Delta(z)=K$ for each $z\in \operatorname{Dom}(\Delta)$. We denote by $P(Z,Y)$ the set of all proper closed convex processes $\Delta$ from $Z$ into $Y$. According to \cite[Definition~2.1.3]{aubin1990set}, the norm of a process $\Delta\in P(Z,Y)$ is equal to
\begin{equation} \label{def:norma_proceso}
\|\Delta\|=\sup_{z\in Z} \inf_{y\in \Delta(z)}\frac{\|y\|}{\|z\|}=\sup_{z\in Z}\frac{d(0_Y,\Delta(z))}{\|z\|} =\sup_{z\in B_Z}\inf_{y\in \Delta(z)}\|y\|=\sup_{z\in S_Z} d(0_Y,\Delta(z)).
\end{equation}

Let $Y_+\subset Y$ be a convex cone. For arbitrary $y_1$, $y_2\in Y$, we write $y_1\leq y_2$ if and only if $y_2-y_1\in Y_+$. Then $\leq$ defines a reflexive and transitive relation (a preorder) on $Y$, and the cone $Y_+$ is called the \emph{ordering cone} on $Y$. We say that a point $y_0\in Y$ is \emph{nondominated} by a set $A\subset Y$, written $y_0 \in \mathrm{ND}(A \mid Y_+)$, if $A \cap (y_0-Y_+) \subset y_0+Y_+$. We say that a point $y_0 \in Y$ is a \emph{minimal point} of a set $A\subset Y$, written $y_0\in \mathrm{Min}(A\mid Y_+)$, if $y_0 \in \mathrm{ND}(A \mid Y_+)$ and $y_0\in A$. When the ordering cone $Y_+$ is clear from the context, we simply write $y_0\in \mathrm{ND}(A)$ and $y_0\in \mathrm{Min}(A)$, respectively.

If $Y_+$ is pointed, then the relation $\leq$ becomes an order, and $y_0\in \mathrm{Min}(A)$ if and only if $A \cap (y_0-Y_+) =\{y_0\}$. Note that in the real line with the usual order, minimal points become minima, and any nondominated point belonging to the closure of a set becomes its infimum. 

Let $K\subset Y$ be a cone. A functional $p:Y\rightarrow \mathbb{R}$ is said to be \emph{$K$-monotone} (respectively, \emph{strictly $K$-monotone}) if $p(y_1)\geq p(y_2)$ (respectively, $p(y_1)> p(y_2)$) for every $y_1$, $y_2\in Y$ such that $y_1\in y_2 +K$ (respectively, $y_1\in y_2 +(K\setminus \{0_Y\})$).

\section{The Optimization Problem and the General Lagrange Duality Framework} \label{sec:general_framework}

In this section, we introduce the parametric set-valued optimization problem under study and develop the main theoretical framework. We define the notion of a Lagrange multiplier as a closed convex process, state the key geometric assumptions that enable our theory (Assumption~\ref{hipotesisGeneral}), and prove the fundamental multiplier theorem (Theorem~\ref{ThLagMult11}). We then specialize these results to the scalar case, establishing an explicit connection with sublinear penalty functions.

Hereafter, we consider three normed spaces $X$, $Y$, and $Z$. The space $Y$ is endowed with a non-trivial ordering cone $Y_+ \subset Y$. We focus on the following family of parametric set-valued optimization problems:
\[
\text{Minimize } F(x) \quad \text{subject to } x \in \Omega,\ z \in G(x), \quad (P(z)), \ z \in U,
\]
where $F:\Omega \rightrightarrows Y$ and $G:\Omega \rightrightarrows Z$ are set-valued mappings, and $U \subset Z$ denotes an open neighborhood of the origin $0_Z$. We define the set-valued map
\[
V: U \rightrightarrows Y, \quad V(z) := F \circ G^{-1}(z) = F(\{x \in \Omega : z \in G(x)\}) = \bigcup_{\substack{x \in \Omega \\ z \in G(x)}} F(x).
\]
It is not difficult to prove the equality
\[
\operatorname{Gph(V)}= \bigcup_{\substack{z \in U \\ x \in \Omega \cap G^{-1}(z)}} (G(x)\times F(x)),  
\]
where we adopt the convention
\[
\varnothing \times A = 
A \times \varnothing = 
\varnothing \times \varnothing = 
\varnothing,
\quad \text{for every set } A.
\]
We shall assume throughout that 
\begin{equation}\label{Hipotesis_regularidad}
V(z) \neq \varnothing \quad \text{for every } z \in U,
\end{equation}
so that the domain of \(V\) coincides with \(U\).

\begin{remark}\label{remark:problema_P'}
The program
\[
\text{Minimize } F(x) \quad \text{subject to } x \in \Omega, \ G(x) \cap (z - Z_+) \neq \varnothing, \quad (P'(z)), \ z \in U
\]
where $Z_+ \subset Z$ is a non-trivial ordering cone, is a standard formulation commonly found in the literature; for instance, in our previous work \cite{GC-Melguizo2022} we studied a Lagrange duality theory for this type of problem under the assumption of convexity. In the present work, however, convexity is not required. 

Moreover, \((P'(z))\) can be expressed in the form of $(P(z))$ by replacing $G$ with the set-valued map $G + Z_+$, defined by $(G+Z_+)(x) := G(x) + Z_+$. Indeed, $G(x) \cap (z - Z_+) \neq \varnothing$ if and only if $z \in G(x) + Z_+$. 

Finally, while $(P'(z))$ relies explicitly on the ordering cone $Z_+$ of $Z$, the formulation $(P(z))$ does not require $Z$ to be an ordered space. 
This provides a conceptual advantage, as our framework can handle set-valued constraints in general normed spaces, not necessarily endowed with an ordering.
\end{remark}

We now particularize the above notions of nondominated and minimal points to the 
context of the vector optimization problem $(P(z))$. In this setting, the relevant set is 
the value set $V(z)$, and we introduce the following terminology.

\begin{definition} \label{InfumablePoint}
Let $z \in U\subset Z$ and consider the corresponding program $(P(z))$. 
A point $y_z \in Y$ is said to be 
\begin{itemize}
    \item[(i)] a nondominated point of $(P(z))$, written $y_z \in \mathrm{ND}(P(z))$, if $y_z$ is nondominated by the set $V(z)$;
    \item[(ii)] a minimal point of $(P(z))$, written $y_z \in \mathrm{Min}(P(z))$, if it is nondominated and belongs to $V(z)$.
\end{itemize}
Unlike  minimal points, nondominated points are not necessarily reached at a feasible solution.
\end{definition}

The key step in our approach is to extend the notion of Lagrange multipliers from the classical linear setting to the framework of closed convex processes. 
This allows us to capture duality properties in the nonconvex context considered here. 
The following definition formalizes what we mean by a Lagrange multiplier in our setting.

\begin{definition} \label{LMs}
Let $y_0\in \mathrm{ND}(P(0_Z))$. A proper closed convex process $\Delta: Z\rightrightarrows Y$ is said to be a Lagrange multiplier of $\mathrm{ND}(P(0_Z))$ at $y_0$, if $y_0$ is a nondominated point of the program
\[
\text{  Min \,} F(x)+\Delta(G(x)) \text{ \  such that \ } x\in \Omega.\text{ \ \ \ \ \ \ } (P[\Delta])
\]
\end{definition}

The following assumption plays a central role in our framework by providing the structural conditions under which our Lagrangian theory can be developed. 

\begin{asumption} \label{hipotesisGeneral}
Given the family of programs $(P(z))$ with $z \in U\subset Z$, and $y_0 \in \mathrm{ND}(P(0_Z))$, assume that there exists a process $\Delta \in P(Z,Y)$ such that 
\begin{itemize}
\item[(a)] $\operatorname{core}(\operatorname{Gph}(\Delta))\neq \varnothing$.
\item[(b)] $Y_+ \subseteq \Delta (0_Z)$ and $(-Y_+) \cap \Delta (0_Z)\subseteq Y_+$.
\item[(c)] $\operatorname{Gph}(-\Delta) \cap [\operatorname{Gph}(V)-(0,y_0)]\subseteq \{ 0_{Z \times  Y} \}$. 
\end{itemize}
\end{asumption}
Geometrically, condition (a) requires the graph of the process $\Delta$ to be sufficiently rich, in the sense that it possesses a nonempty algebraic interior. 
This prevents $\Delta$ from being too “thin’’ and ensures that it has enough directions to play the role of a multiplier. 
Condition (b) imposes a compatibility between $\Delta$ and the ordering cone $Y_+$: 
the image of $0_Z$ through $\Delta$ must contain all nonnegative directions, while its intersection with the opposite cone is restricted to directions already belonging to $Y_+$. Finally, condition (c) encodes a strict separation property between the  graph of $-\Delta$ and the graph of $V$ shifted by $(0,y_0)$. Indeed, the graph of $-\Delta$ lies entirely apart from the translation of the graph of $V$ by $(0, y_0)$, except possibly at the origin. In other words, no nontrivial direction of $V$ through $y_0$ can penetrate the graph of $-\Delta$. This geometric separation ensures that $-\Delta$ acts as a genuine supporting structure, delimiting the feasible region associated with $V$ at the point $y_0$ and preventing any mutual overlap beyond the trivial intersection at the origin. Together, these requirements provide the geometric backbone for the Lagrange framework developed below.

In order to streamline the presentation and to emphasize the central role played by processes $\Delta$ fulfilling Assumption~\ref{hipotesisGeneral}, 
it is convenient to collect them into a single set.

\begin{definition}\label{defi_Gamma_y_0} 
We denote by $\Gamma_{y_0}$ the set of all processes $\Delta\in P(Z,Y)$ that satisfy conditions \((a)\)–\((c)\) in Assumption \ref{hipotesisGeneral}. 
Equivalently, 
\[
\Gamma_{y_0} := \{\Delta \in P(Z,Y) \;|\; \Delta \text{ verifies (a)–(c) in Assumption \ref{hipotesisGeneral}}\}.
\]
\end{definition}

The next result establishes the fundamental role of the set $\Gamma_{y_0}$ in our Lagrange duality framework. 
It provides a concrete link between the abstract conditions of Assumption~\ref{hipotesisGeneral} 
and the dual optimization problem $(P[\Delta])$. 
In particular, this theorem guarantees that the minimality or nondominance of $y_0$ in the primal problem is preserved under the corresponding dual construction.

\begin{theorem} \label{ThLagMult11}
Let $y_0 \in \mathrm{ND}(P(0_Z))$ and assume that $\Gamma_{y_0}\not = \varnothing$. 
Then, every $\Delta \in \Gamma_{y_0}$ is a Lagrange multiplier of $(P(0_Z))$ at $y_0$; 
that is, $y_0$ is a nondominated point of the program
\begin{equation*}\label{DualPr_b'}
\text{Minimize } F(x) + \Delta(G(x)) \quad \text{subject to } x \in \Omega. 
\tag{$P[\Delta]$}
\end{equation*}

Furthermore, if $y_0$ is a minimal point of $(P(0_Z))$ (that is, if $y_0 \in F(x_0)$ 
for some feasible solution $x_0$), then $y_0$ is also a minimal point of $(P[\Delta])$ 
(achieved at $x_0$), and the compatibility condition
\begin{equation}\label{InterseccionTeo_b'}
\Delta(G(x_0)) \cap (-Y_+) \subseteq Y_+
\end{equation}
holds.
\end{theorem}

\begin{remark}
In the construction of the above theorem, the convex process $\Delta$ acts as a
separation device between feasible points and directions of improvement.
In this sense, $\Delta$ fulfills the same structural role as continuous linear
functionals in the classical scalar convex setting, encoding feasibility and
optimality through a separation principle; see, for example \cite[Section~8.4, Theorem 1]{Luenberger1969}.
\end{remark}

To prove Theorem~\ref{ThLagMult11}, we will rely on the following separation result, which combines both topological and algebraic features of convex cones in the setting of locally convex spaces. It provides the existence of a functional that strictly separates elements inside a cone from those outside, a tool that will be crucial in constructing the Lagrange multipliers.

\begin{theorem} \label{thm:separacion_Gerth}
Let $Y$ be a locally convex space, and let $K \subset Y$ be a closed convex cone with $\operatorname{core}(K) \neq \varnothing$. 
Then there exists a $K$-monotone sublinear functional $p: Y \to \mathbb{R}$ such that
\[
p(x) \leq 0 < p(y), \quad \text{for all } x \in -K \text{ and } y \notin -K.
\]
\end{theorem}

\begin{proof}
Let us denote by $\tau$ the original locally convex topology on $Y$, and by $\tau_c$ the convex core topology on $Y$, that is, the topology generated by the family of all seminorms defined on $Y$. Then, $\tau_c$ is the strongest locally convex topology on $Y$, so the cone $K$ is $\tau_c$-closed, and the set $Y \setminus (-K)$ is $\tau_c$-open (we refer the reader to \cite[Section~6.3]{KTZ2015} for a comprehensive study of the topology $\tau_c$). Therefore, by \cite[Proposition~6.3.1(iii)]{KTZ2015}, it follows that $\operatorname{core}(K) = \operatorname{int}_{\tau_c}(K)$. Applying now \cite[Corollary~2.1]{Gerth90} to the set $A := Y \setminus (-K)$ and the cone $K$, we obtain the existence of a $\tau_c$-continuous sublinear functional $p : Y \to \mathbb{R}$ that is strictly $\operatorname{int}_{\tau_c}(K)$-monotone and satisfies
\[
p(x) \leq 0 < p(y), \quad \text{for all } x \in -K \text{ and } y \notin -K.
\]
Finally, applying \cite[Lemma~2.4]{KTZ2015}, we conclude that $p$ is $K$-monotone.
\end{proof}

\begin{proof}[Proof of Theorem \ref{ThLagMult11}]
Assume first that $y_0 = 0_Y$. Then, we suppose that $0_Y \in \mathrm{ND}(P(0_Z))$, which is equivalent to
\[
V(0_Z) \cap (-Y_+) \subset Y_+,
\]
and fix some $\Delta \in \Gamma_{y_0}$. We aim to show that
\begin{equation}\label{eq:inclusion_1_prueba_tma_general}
\left( \bigcup_{x \in \Omega} (F(x) + \Delta(G(x))) \right) \cap (-Y_+) \subset Y_+.    
\end{equation}

Let $x \in \Omega$ be arbitrary, with $z \in G(x)$ and $y_1 \in F(x)$, $y_2 \in \Delta(z)$. Define $y := y_1 + y_2$. Then, $y \in V(z) + \Delta(z)$. Observe that
\[
(z, y_2) \in \mathrm{Gph}(\Delta) \;\Leftrightarrow\; (z, -y_2) \in \mathrm{Gph}(-\Delta) \;\Leftrightarrow\; (-z, y_2) \in -\mathrm{Gph}(-\Delta).
\]
Set $K := -\mathrm{Gph}(-\Delta)$. By Assumption~\ref{hipotesisGeneral}, $-K = \mathrm{Gph}(-\Delta)$ is a closed convex cone with nonempty core, $\mathrm{Gph}(V) \setminus \{0_{Z \times Y}\} \subset (Z \times Y) \setminus \mathrm{Gph}(-\Delta)$, and the latter set is open. Then, Theorem~\ref{thm:separacion_Gerth} applied to $K$ ensures the existence of a $K$-monotone sublinear functional $p: Z \times Y \to \mathbb{R}$ such that:
\begin{equation}\label{desigualdad_sublineal}
p(z, y) > 0 \quad \text{for all } (z, y) \in \mathrm{Gph}(V) \setminus \{0_{Z\times Y}\},
\end{equation}
and
\begin{equation}\label{desigualdad1}
 p(z', y') \leq 0 \quad \text{for all } (z', y') \in \mathrm{Gph}(-\Delta).
\end{equation}

We now consider two cases:

\begin{itemize}
    \item[\textbf{Case 1:}] $(z, y_1) \neq 0_{Z \times Y}$. Since $(-z, y_2) \in K$ and $p$ is monotone on $K$, we get
    \[
    p(0_Z, y_1 + y_2) = p\big((z, y_1) + (-z, y_2)\big) \geq p(z, y_1) > 0,
    \]
    where the last inequality follows from \eqref{desigualdad_sublineal}. Thus, $(0_Z, y_1 + y_2) \notin \mathrm{Gph}(-\Delta)$, which implies $y = y_1 + y_2 \notin -\Delta(0_Z)$. By the first inclusion of Assumption~\ref{hipotesisGeneral} (b), this yields $y \notin -Y_+$, i.e., $$y \notin \left( \bigcup_{x \in \Omega} (F(x) + \Delta(G(x))) \right) \cap (-Y_+).$$
    
    \item[\textbf{Case 2:}] If $(z, y_1) = 0_{Z \times Y}$, then $z = 0_Z$ and $y_1 = 0_Y$, so $y = y_2 \in \Delta(0_Z)$. If 
\[
y \in \left( \bigcup_{x \in \Omega} (F(x) + \Delta(G(x))) \right) \cap (-Y_+),
\]
then the second inclusion of  Assumption~\ref{hipotesisGeneral} (b) implies that $y \in Y_+$. This completes the proof of~\eqref{eq:inclusion_1_prueba_tma_general}.
\end{itemize}

To prove the final part of the statement, assume there exists $x_0 \in \Omega$ such that $0_Y \in F(x_0)$ and $0_Z \in G(x_0)$. We claim that $0_Y \in F(x_0) + \Delta(G(x_0))$. Indeed, since $(0_Z, 0_Y) \in \mathrm{Gph}(\Delta)$, we have $0_Y \in \Delta(0_Z)$. Then, as $0_Z \in G(x_0)$, it follows that $0_Y \in \Delta(G(x_0))$, and thus $0_Y \in F(x_0) + \Delta(G(x_0))$.

Finally, let us prove inclusion~\eqref{InterseccionTeo_b'}. Take $u \in \Delta(G(x_0)) \cap (-Y_+)$. Since $0_Y$ is a minimal point of $(P[\Delta])$ achieved at $x_0$, we have that $0_Y \in F(x_0)$ and $\left( F(x_0) + \Delta(G(x_0)) \right) \cap (-Y_+) \subset Y_+$. Thus, clearly $u \in Y_+$.

\medskip

To finish, consider the case $y_0 \neq 0_Y$. Then, we have:
\[
y_0 \in \mathrm{ND}(P(0_Z)) \;\Leftrightarrow\; V(0_Z) \cap (y_0 - Y_+) \subset y_0 + Y_+ 
\;\Leftrightarrow\; (V(0_Z) - y_0) \cap (-Y_+) \subset Y_+
\]
\[
\Leftrightarrow \left( \bigcup_{\substack{x \in \Omega \\ 0_Z \in G(x)}} (F(x) - y_0) \right) \cap (-Y_+) \subset Y_+
\]
\[
\overset{(*)}{\Rightarrow} \left( \bigcup_{x \in \Omega} (F(x) - y_0 + \Delta(G(x))) \right) \cap (-Y_+) \subset Y_+
\]

\[
\Leftrightarrow \left( \bigcup_{x \in \Omega} (F(x) + \Delta(G(x))) \right) \cap (y_0 - Y_+) \subset y_0 + Y_+
\]
\[\Rightarrow y_0 \mbox{ is a nondominated point of }(P[\Delta]),\]
where implication $(*)$ follows from the proof of the case $y_0 = 0_Y$, by replacing $F(x)$ with $F(x) - y_0$.

For the final part of the statement, assume that $y_0 \in F(x_0)$ for some feasible solution $x_0$. Then, $0_Y \in F(x_0) - y_0$, and by the proof for the case $y_0 = 0_Y$, we conclude that $0_Y$ is a minimal point of the problem
\[
\text{Minimize } F(x) - y_0 + \Delta(G(x)) \text{ subject to } x \in \Omega,
\]
which implies that $y_0$ is a minimal point of $(P[\Delta])$. A straightforward adaptation of the proof of~\eqref{InterseccionTeo_b'} for the case $y_0 = 0_Y$ yields the conclusion for the general case.
\end{proof}

We present two examples. The first one, in finite dimension, illustrates the conclusion of Theorem~\ref{ThLagMult11}. The second example, set in the space $\ell^2$, shows that the conclusion of Theorem~\ref{ThLagMult11} may still hold even when condition~(a) of Assumption~\ref{hipotesisGeneral} fails. Consequently, this condition provides a sufficient, but not necessary, requirement for the existence of Lagrange multiplier processes within the framework developed in this paper.

\begin{example}\label{ex:R2_multiplier_compact}
Let $X=Y=\mathbb{R}^2$ and $Z=\mathbb{R}$, with ordering cone $Y_+=\mathbb{R}^2_+$.  
Consider the feasible set
\[
\Omega:=\{(x_1,x_2)\in\mathbb{R}^2 : x_1^2+x_2^2\le 1,\ x_2\ge 0\},
\]
and define the mappings
\[
F(x_1,x_2):=\{(x_1^2+x_2^2,\; x_2^2+x_1x_2)\}, 
\qquad 
G(x_1,x_2):=\{x_1\}.
\]

For $z=(0,0)$, one has
\[
G^{-1}(0)=\{(0,x_2):0\le x_2\le 1\}.
\]
The point $x_0=(0,0)$ satisfies $F(x_0)=\{0_Y\}$, while
$F(x)\subset Y_+\setminus\{0_Y\}$ for all $x\in G^{-1}(0)\setminus\{x_0\}$.
Hence $(0,0)$ is a minimal value of $(P(0))$.

Define the process $\Delta:Z\rightrightarrows Y$ by
\[
\Delta(z):=\{y\in\mathbb{R}^2:\ y_1\ge |z|,\ y_2\ge |z|\},\qquad z\in\mathbb{R}.
\]
Then $\Delta\in P(Z,Y)$ and $\operatorname{core}(\operatorname{Gph}(\Delta))\neq\varnothing$.
Moreover,
\[
F(x_0)+\Delta(G(x_0))=Y_+,
\]
while $(0,0)\notin F(x)+\Delta(G(x))$ for every $x\in\Omega\setminus\{x_0\}$.
Therefore, $x_0$ is also a minimal solution of the penalized problem $(P[\Delta])$.
\end{example}

\begin{example}\label{ex:ell2_BZ}

Let $X=Y=Z:=\ell^2$, ordered by the positive cone
\[
Y_+ := \{y\in\ell^2 : y_i\ge 0 \ \forall i\in\mathbb{N}\},
\]
for which $\operatorname{core}(Y_+)=\varnothing$.
Let
\[
\Omega := \{x\in\ell^2 : \|x\|_{\ell^2}\le 1\},
\]
and define
\[
F(x):=\{(x_i^2)_{i\in\mathbb{N}}\}+Y_+,
\qquad
G(x):=\{(x_i^3)_{i\in\mathbb{N}}\}.
\]

Fix $z=0_{\ell^2}$. Since
\[
G^{-1}(0_{\ell^2})=\{x\in\Omega : x_i^3=0 \ \forall i\}=\{0_{\ell^2}\},
\]
we obtain
\[
F\circ G^{-1}(0_{\ell^2})=F(0_{\ell^2})=Y_+,
\]
and hence $0_{\ell^2}$ is a minimal solution of the parametric problem $(P(0))$.

Define the multiplier mapping $\Delta:Z\rightrightarrows Y$ by
\[
\Delta(z):=\{y\in\ell^2 : y_i\ge \|z\|_{\ell^2} \ \forall i\in\mathbb{N}\}.
\]
Its graph
\[
\operatorname{Gph}(\Delta)
=
\{(z,y)\in\ell^2\times\ell^2 : y_i\ge \|z\|_{\ell^2} \ \forall i\}
\]
is a closed convex cone, since the norm on $\ell^2$ is continuous, convex and positively homogeneous.
Therefore, $\Delta$ is a convex process.

For any $x\in\Omega$, using $\Delta(z)\subset Y_+$ for all $z\in\ell^2$, we have
\[
F(x)+\Delta(G(x))
=
\{(x_i^2)_{i\in\mathbb{N}}\}+Y_+.
\]
In particular,
\[
F(0_{\ell^2})+\Delta(G(0_{\ell^2}))=Y_+.
\]
If $x\neq 0_{\ell^2}$, then $(x_i^2)_{i\in\mathbb{N}}\in Y_+\setminus\{0\}$, and hence
\[
(F(x)+\Delta(G(x)))\cap (-Y_+)\subset \{ 0_{\ell^2}\},
\]
showing that the penalized problem admits $0_{\ell^2}$ as a minimal solution.

Finally, since $\Delta(z)\subset Y_+$ for all $z\in\ell^2$, we have
\[
\operatorname{core}(\Delta(z))=\varnothing
\quad\text{and hence}\quad
\operatorname{core}(\operatorname{Gph}(\Delta))=\varnothing.
\]
Therefore, the conclusion of Theorem~\ref{ThLagMult11} remains valid in this infinite-dimensional setting, and exact penalization holds despite the failure of the interiority condition.
\end{example}

In the remainder of this section, we focus on the particular case \( Y = \mathbb{R} \) and \( Y_+ = \mathbb{R}_+ \). In this setting, the structure of the real line enables a reformulation of Assumption~\ref{hipotesisGeneral} in terms of sublinear functions, yielding a considerably simpler expression. 
The results obtained here naturally complement the classical global optimization theorems for convex functionals, such as those in \cite[Chapters~7, 8]{Luenberger1969}. 

Next, we show that every element of \( \Gamma_{y_0} \subset P(Z,\mathbb{R}) \)—relevant to our Lagrangian framework—can be identified with a lower semicontinuous sublinear function, providing a more transparent dual representation. 
In this context, the three conditions in Assumption~\ref{hipotesisGeneral} reduce to a single one—analogous to condition~(c)—expressed in terms of these sublinear functions. 
The corresponding family is denoted by \( \mathcal{S}_{r_0} \) and defined below.

\begin{definition}\label{defi_Gamma_r_0_real} 
Assume that \( Y = \mathbb{R} \) and \( Y_+ = \mathbb{R}_+ \), and let \( r_0 \in \mathbb{R} \). 
We denote by \( \mathcal{S}_{r_0} \) the set of all lower semicontinuous sublinear functions 
\( \varphi : Z \to \mathbb{R} \) that satisfy, for every \( z \in Z \setminus \{0_Z\} \),
\begin{equation} \label{condicion_caso_real}
    -\varphi(z) < r - r_0 \quad \text{for all } r \in F(G^{-1}(z)).
\end{equation}
Equivalently,
\[
\mathcal{S}_{r_0} := 
\Bigl\{\, 
\varphi : Z \to \mathbb{R} \;\Big|\;
\begin{aligned}[t]
& \varphi \text{ is sublinear and lower semicontinuous},\\
& \operatorname{dom}(\varphi) = Z,\\
& \text{and } \varphi \text{ satisfies } \eqref{condicion_caso_real}
\text{ for all } z \in Z \setminus \{0_Z\}
\Bigr\}.
\end{aligned}
\]
\end{definition}

We now establish a one-to-one correspondence between the elements of $\mathcal{S}_{r_0}$  and $\Gamma_{r_0}$, showing that the processes satisfying Assumption~\ref{hipotesisGeneral} can be equivalently described in terms of sublinear functions.

\begin{proposition}\label{proposition:biyeccion_Gamma_Gamma'}
Assume that $Y = \mathbb{R}$ and $Y_+ = \mathbb{R}_+$. Consider the family of programs $(P(z))$ with $z \in U \subset Z$, and let $r_0 \in \mathrm{ND}(P(0_Z))$. Let us consider the mapping
\[
\Upsilon : \mathcal{S}_{r_0} \longrightarrow \Gamma_{r_0},
\]
which assigns to each sublinear function $\varphi \in \mathcal{S}_{r_0}$ the process $\Upsilon(\varphi)\in P(Z,\mathbb{R})$ defined by
\begin{equation}\label{eq_defi_Upsilon}
\operatorname{Gph}(\Upsilon(\varphi)) := \operatorname{epi}(\varphi).
\end{equation}
Then the following statements hold.
\begin{itemize}
    \item[(i)] The mapping $\Upsilon$ is well defined; that is, $\Upsilon(\varphi) \in \Gamma_{r_0}$ for every $\varphi \in \mathcal{S}_{r_0}$.
    \item[(ii)] The mapping $\Upsilon$ is bijective.
\item[(iii)] If {\small $\|\Upsilon(\varphi)\| < +\infty$} (for instance, when $Z$ is a Banach space), then \\ $\operatorname{int}(\operatorname{Gph}(\Upsilon(\varphi))) \neq \varnothing$.
\end{itemize}
\end{proposition}
\begin{proof}
(i) Let \( \varphi \in \mathcal{S}_{r_0} \). We will prove that the process \( \Upsilon(\varphi) \) belongs to \( \Gamma_{r_0} \) by verifying that it satisfies conditions (a)--(c) in Assumption~\ref{hipotesisGeneral}. As \( \operatorname{dom}(\varphi) = Z \), it is clear that \( \operatorname{dom}(\Upsilon(\varphi)) = Z \). On the one hand, \cite[Prop.~2.3]{EkelandTemam1976} states that a sublinear functional \( \varphi : Z \to \mathbb{R} \) is lower semicontinuous if and only if the corresponding process \( \Upsilon(\varphi) \) has a closed graph. Recall that a process \( \Delta \) is said to be closed-valued when $\Delta(x)$ is closed for every $x \in \operatorname{dom} \Delta$.  Note that a process with a closed graph is necessarily closed-valued, since each fiber $\{x\} \times \Delta(x)$ is the intersection of the closed graph of $\Delta$ with the closed set $\{x\} \times \mathbb{R}$. On the other hand, as $\varphi$ is sublinear, it is convex, and so $\operatorname{epi}(\varphi)$ is convex \cite[Prop.~2.1]{EkelandTemam1976}. As a consequence, $\Upsilon(\varphi)$ is a real, proper, convex process with closed values. Furthermore, \cite[Prop.~1.3]{Abreu89} establishes that every real proper convex process $\Delta: Z \rightrightarrows \mathbb{R}$ with closed values is positive, i.e., $\Upsilon(\varphi)(0_Z)=\mathbb{R}_+$,
if and only if there exists a sublinear mapping $\varphi: Z \to \mathbb{R}$ such that \(\operatorname{Gph}(\Delta) = \operatorname{epi}(\varphi)\). Therefore, $\Upsilon(\varphi)\in P(Z,\mathbb{R})$ and $\Upsilon(\varphi)(0_Z)=\mathbb{R}_+$. Next, we verify that $\Upsilon(\varphi)$ satisfies conditions (a)--(c) of Assumption \ref{hipotesisGeneral}. Condition (a) is $\operatorname{core}(\operatorname{Gph}(\Upsilon(\varphi)))\not = \varnothing$, and we will check that $(0_Z,2) \in \operatorname{core}(\operatorname{Gph}(\Upsilon(\varphi)))$ showing that for every $(\bar{z},\bar{r}) \in Z\times \mathbb{R}$, there exists $\varepsilon>0$ such that  
\[
(0_Z,2)+\zeta (\bar{z},\bar{r}) \in  \operatorname{Gph}(\Upsilon(\varphi)), 
\]
for every $0\le\zeta\le \varepsilon$. To this end, fix an arbitrary $(\bar{z},\bar{r}) \in Z\times \mathbb{R}$. Choose $r_{\bar{z}}\ge \max\{\varphi(\bar{z}),1\}$ and set $\tau_{\bar{z}}:=1/r_{\bar{z}}\in (0,1]$. Clearly, $(\bar{z},r_{\bar{z}})\in \operatorname{Gph}(\Upsilon(\varphi))=\operatorname{epi}(\varphi)$. We claim that 
\begin{equation}\label{eq_core_Upsilon}
 (t\bar{z},s)\in \operatorname{Gph}(\Upsilon(\varphi)), \quad \text{for every } 1\le s\le 3,\; 0<t\le \tau_{\bar{z}}.   
\end{equation}
Indeed, given any $0<t\le \tau_{\bar{z}}$, since $\operatorname{Gph}(\Upsilon(\varphi))$ is a cone, we have $(t\bar{z},tr_{\bar{z}})\in \operatorname{Gph}(\Upsilon(\varphi))$. Moreover, because $\Upsilon(\varphi)$ is a positive process and $s-tr_{\bar{z}}\ge 0$, it follows that $(0_Z,s-tr_{\bar{z}})\in \operatorname{Gph}(\Upsilon(\varphi))$. By convexity, we then obtain
\[
(t\bar{z},s)
  =(t\bar{z},tr_{\bar{z}})+(0_Z,s-tr_{\bar{z}})
  \in \operatorname{Gph}(\Upsilon(\varphi))+\operatorname{Gph}(\Upsilon(\varphi))
  \subset \operatorname{Gph}(\Upsilon(\varphi)),
\]
for every $1\le s\le 3$, which proves \eqref{eq_core_Upsilon}.

Now choose $0<\varepsilon\le \tau_{\bar{z}}$ small enough so that $-1\le \zeta \bar{r} \le 1$ for all $0\le \zeta \le \varepsilon$. Hence $1\le 2+\zeta\bar{r}\le 3$ for all such $\zeta$. Consequently,
\[
(0_Z,2)+\zeta(\bar{z},\bar{r})=(\zeta \bar{z},2+\zeta \bar{r})\in \operatorname{Gph}(\Upsilon(\varphi)), \quad \text{for every } 0\le \zeta \le \varepsilon.
\]
Hence, condition~(a) in Assumption~\ref{hipotesisGeneral} holds. We next verify condition~(b). Since $\Upsilon(\varphi)(0_Z) = \mathbb{R}_+$, we have
\[
Y_+ = \mathbb{R}_+ = \Upsilon(\varphi)(0_Z),
\qquad
(-Y_+) \cap \Upsilon(\varphi)(0_Z) = \{0\} \subset \mathbb{R}_+,
\]
and therefore condition~(b) is also satisfied.

Finally, we check condition~(c). Suppose, by contradiction, that there exists
\[
(\bar{z},\bar{r}) \in \operatorname{Gph}(-\Upsilon(\varphi)) 
\cap [\operatorname{Gph}(V) - (0_Z, r_0)],
\qquad (\bar{z},\bar{r}) \ne (0_Z,0).
\]
Then there exist points $z' \in Z$ and $r' \in F(G^{-1}(z'))$ such that 
\[
(\bar{z},\bar{r}) = (z', r' - r_0) \in 
\operatorname{Gph}(-\Upsilon(\varphi))
= \{(z,r) \in Z \times \mathbb{R} : -\varphi(z) \ge r\}.
\]
Hence $-\varphi(z') \ge r' - r_0$, which contradicts \eqref{condicion_caso_real} unless $z' = 0_Z$. 
If $z' = 0_Z$, then $r' \in V(0_Z)$, and since $\varphi(0_Z)=0$, we obtain  $r' \ge r_0$. 
However, as $r_0 \in \mathrm{ND}(P(0_Z))$, it follows that $r_0 \le r'$, 
and thus $r' = r_0$, which implies $(\bar{z},\bar{r}) = (0_Z,0)$, a contradiction. 
Therefore, no such $(\bar{z},\bar{r})$ exists, and condition~(c) in Assumption~\ref{hipotesisGeneral} is verified. 

Therefore, we conclude that $\Upsilon(\varphi)\in \Gamma_{r_0}$; hence, $\Upsilon$ is well defined.

\medskip

(ii) It is clear that $\Upsilon$ is injective because, for any $\varphi_1, \varphi_2 \in \mathcal{S}_{r_0}$,
\[
\varphi_1 = \varphi_2 \quad \Longleftrightarrow \quad \operatorname{epi}(\varphi_1) = \operatorname{epi}(\varphi_2).
\]

Next, we show that $\Upsilon$ is surjective.  For that purpose, we assume that $\Delta \in \Gamma_{r_0}$, i.e., $\Delta \in P(Z,\mathbb{R})$ satisfies Assumption~\ref{hipotesisGeneral} with $r_0$ playing the role of $y_0$. 
We first claim that $\Delta$ is positive, i.e., $\Delta(0_Z)=\mathbb{R}_+$. Indeed, by condition~(b) in Assumption~\ref{hipotesisGeneral} we have $\mathbb{R}_+ \subset \Delta(0_Z)$ and the last set is a cone. Hence, either $\Delta(0_Z)=\mathbb{R}_+$ or $\Delta(0_Z)=\mathbb{R}$. 
The latter is impossible because condition~(b) guarantees that
\[
(-\mathbb{R}_+)\cap\Delta(0_Z)\subset\mathbb{R}_+.
\]
Indeed, if $\Delta(0_Z)=\mathbb{R}$, then $(-\mathbb{R}_+)\cap\Delta(0_Z)= -\mathbb{R}_+$, and the inclusion above would give $-\mathbb{R}_+\subset\mathbb{R}_+$, which is false. 
Therefore $\Delta(0_Z)=\mathbb{R}_+$. Now, again by \cite[Prop.~1.3]{Abreu89} and \cite[Prop.~2.3]{EkelandTemam1976}, there exists a lower semicontinuous sublinear mapping $\varphi: Z \to \mathbb{R}$ such that \(\operatorname{Gph}(\Delta) = \operatorname{epi}(\varphi)\). We will verify that $\varphi\in \mathcal{S}_{r_0}$ by showing that $\varphi$ satisfies \eqref{condicion_caso_real}. Fix $z\in Z\setminus\{0_Z\}$ and an arbitrary $r\in F(G^{-1}(z))$. 
Recalling $V:=F\circ G^{-1}$, we have $(z,r)\in\operatorname{Gph}(V)$. 
By condition~(c) in Assumption~\ref{hipotesisGeneral} it follows that
\[
(z,\,r-r_0)\not\in\operatorname{Gph}(-\Delta).
\]

Since $\operatorname{Gph}(\Delta)=\operatorname{epi}(\varphi)=\{(z,t)\in Z\times\mathbb{R} : \varphi(z)\le t\}$, we have
\[
\operatorname{Gph}(-\Delta)=\{(z,s)\in Z\times\mathbb{R} : (z,-s)\in\operatorname{Gph}(\Delta)\}
=
\]
\[
\{(z,s)\in Z\times\mathbb{R} : \varphi(z)\le -s\}=\{(z,s)\in Z\times\mathbb{R} : -\varphi(z)\ge s\}.
\]
Therefore, $(z,r-r_0)\not\in\operatorname{Gph}(-\Delta)$ is equivalent to
\[
-\varphi(z) < r-r_0,
\]
which is exactly \eqref{condicion_caso_real}. 

\medskip

(iii) Assume that $\|\Upsilon(\varphi)\|< +\infty$. We will show that $(0_Z,2)\in \operatorname{int}(\operatorname{Gph}(\Upsilon(\varphi)))$ by proving that there exists $\varepsilon>0$ such that  
\[
(0_Z,2)+\zeta (\bar{z},\bar{r}) \in  \operatorname{Gph}(\Upsilon(\varphi)), 
\]
for every $0\le \zeta\le \varepsilon$ and  $(\bar{z},\bar{r})\in B_Z\times [-1,1]$.

We first claim that 
\begin{equation}\label{eq:norma_Upsilon}
\|\Upsilon(\varphi)\|=\sup_{z\in S_Z} \varphi(z).     
\end{equation}
Indeed, by \eqref{def:norma_proceso} we have 
\[
\|\Upsilon(\varphi)\|=\sup_{z\in B_Z}\inf \{|y|\colon y\in \Upsilon(\varphi)(z)\}.
\]
Since $\operatorname{Gph}(\Upsilon(\varphi))=\operatorname{epi}(\varphi)$, it follows that 
\[
\|\Upsilon(\varphi)\|=\sup_{z\in B_Z}\inf\{ |y|\colon \varphi(z)\le y\}
   =\sup_{z\in B_Z}\{\max\{\varphi(z),0\}\}
   =\sup_{z\in B_Z}\varphi(z).
\]
To justify the last equality, observe that if $\varphi(z)<0$, then by the sublinearity of $\varphi$ we have $0<-\varphi(z)\le \varphi(-z)$, with $-z\in B_Z$. This proves the claim.

Now, define 
\begin{equation}  
r := \max \{\|\Upsilon(\varphi)\|,1\}, 
\qquad 
\varepsilon:=\frac{1}{r}\in (0,1].
\end{equation}
Fix an arbitrary $(\bar{z},\bar{r}) \in B_Z\times [-1,1]$. Then, by \eqref{eq:norma_Upsilon}, we have $r\ge \varphi(\bar{z})$, which implies $(\bar{z},r)\in \operatorname{Gph}(\Upsilon(\varphi))$. Repeating the argument used in part (ii) to obtain formula \eqref{eq_core_Upsilon}, we deduce that $(t\bar{z},s)\in \operatorname{Gph}(\Upsilon(\varphi))$ for every $1\le s\le 3$ and $0<t\le \varepsilon$. Moreover, since $-1\le \zeta\bar{r}\le 1$ for all $0\le \zeta\le \varepsilon$, repeating the final argument of (ii) yields
\[
(0_Z,2)+\zeta(\bar{z},\bar{r})=(\zeta\bar{z},2+\zeta\bar{r})\in \operatorname{Gph}(\Upsilon(\varphi)), 
\quad \text{for all } 0\le \zeta\le \varepsilon.
\]
This completes the proof, since $\varepsilon$ does not depend on the particular choice of $(\bar{z},\bar{r})\in B_Z\times [-1,1]$.

Finally, note that, by \cite[Theorem~2.2.6]{aubin1990set}, if $Z$ is a Banach space, then $\|\Upsilon(\varphi)\|<+\infty$.
 \end{proof}

\begin{remark}\label{rem:elementos de Gamma positivos}
Note that in the previous proof, we have shown that $\Delta(0_Z) = \mathbb{R}_+$ for every $\Delta \in \Gamma_{r_0}$.
\end{remark}

\begin{theorem} \label{ThLagMult11_real}
Assume that $Y = \mathbb{R}$ and $Y_+ = \mathbb{R}_+$. Let
\[
r_0 := \inf \{ F(x) : x \in \Omega, \;  0_Z \in G(x) \},
\]
assume that $r_0\in \mathbb{R}$ and that $\mathcal{S}_{r_0}\not = \varnothing$. Then, every $\varphi\in \mathcal{S}_{r_0}$ is a Lagrange multiplier of $(P(0_Z))$ at $r_0$; that is,  
$r_0$ is the infimum of the program
\begin{equation*}\label{DualPr_r}
\text{Minimize } F(x) + \varphi(G(x)) \quad \text{subject to } x \in \Omega. \tag{$P[\varphi]$}
\end{equation*}
Furthermore, if $r_0$ is a minimum of $(P(0_Z))$ (that is, if $r_0 \in F(x_0)$ for some feasible solution $x_0$), then $r_0$ is also a minimal point of $(P[\varphi])$ (achieved at $x_0$), and
\begin{equation}\label{InterseccionTeo_real}
\varphi(z) \geq 0  \quad \text{for every } z \in G(x_0).
\end{equation}
\end{theorem}

\begin{proof}
Since $\varphi \in \mathcal{S}_{r_0} \neq \varnothing$, Proposition~\ref{proposition:biyeccion_Gamma_Gamma'} ensures that $\Upsilon(\varphi)\in \Gamma_{r_0} \neq \varnothing$.  
Let $\Delta := \Upsilon(\varphi)$, so that $\Delta(z) = \varphi(z) + \mathbb{R}_+$ and $\operatorname{Gph}(\Delta) = \operatorname{epi}(\varphi)$.  
By Theorem~\ref{ThLagMult11}, $r_0$ is a nondominated point of the program
\begin{equation*}\label{DualPr_b' real}
\text{Minimize } F(x) + \Delta(G(x)) \quad \text{subject to } x \in \Omega. \tag{$P[\Delta]$}
\end{equation*}

If, in addition, $r_0$ is a minimum of $(P(0_Z))$ (that is, $r_0 \in F(x_0)$ for some feasible $x_0$), Theorem~\ref{ThLagMult11} yields that $r_0$ is also a minimum of $(P[\Delta])$, attained at $x_0$, and
\begin{equation}\label{InterseccionTeo_b'_sub}
\Delta(G(x_0)) \cap (-\mathbb{R}_+) \subseteq \mathbb{R}_+.
\end{equation}
Since $Y = \mathbb{R}$ and $Y_+ = \mathbb{R}_+$, minimality and minimal coincide, and \eqref{InterseccionTeo_b'_sub} clearly implies $\Delta(G(x_0)) \subset \mathbb{R}_+$.  
As $\operatorname{Gph}(\varphi) \subset \operatorname{Gph}(\Delta)$, we deduce that \eqref{InterseccionTeo_real} holds.

We now prove that $r_0$ is the infimum of $(P[\varphi])$.  
Define the sets
\[
A := \bigcup_{\substack{x \in \Omega \\ 0_Z \in G(x)}} F(x), \quad  
B := \bigcup_{x \in \Omega} \big(F(x) + \Delta(G(x))\big), \quad  
C := \bigcup_{x \in \Omega} \big(F(x) + \varphi(G(x))\big).
\]
Then $r_0 = \inf A$, and we aim to show that $r_0 = \inf C$.  
Since $\operatorname{ND}(A)$ consists of the lower bounds of $A$, it follows that \(r_0 \in \operatorname{ND}(A)\), and by Theorem~\ref{ThLagMult11}, \(r_0 \in \operatorname{ND}(B)\).  
From Remark~\ref{rem:elementos de Gamma positivos}, \(0 \in \Delta(0_Z) = \mathbb{R}_+\), which implies \(A \subset B\), hence \(r_0 = \inf B\).  
Moreover, $\operatorname{Gph}(\varphi) \subset \operatorname{Gph}(\Delta)$ entails \(C \subset B\) and thus $\inf B \le \inf C$.  
Conversely, since $\Delta(z) = \varphi(z) + \mathbb{R}_+$, for each \(b \in B\) there exists \(c \in C\) with \(c \le b\), giving $\inf C \le \inf B$.  
Therefore, $\inf B = \inf C$, and hence \(r_0 = \inf C\).

Finally, if \(r_0 = \min A\), Theorem~\ref{ThLagMult11} implies \(r_0 = \min B\).  
By the previous argument, \(r_0 = \inf C\); moreover, \(r_0 \in C\), since otherwise there would exist \(c \in C\) with \(c < r_0\), contradicting \(r_0 = \min B\) because \(C \subset B\).
\end{proof}

Let us consider now the case $Y=\mathbb{R}$ and $Y_+=\mathbb{R}_+$ under the situation where both mappings \( F \) and \( G \) involved in the problems $(P(z))$ are single-valued functions. In this case, we have the family of parametric optimization problems:
\[
\text{Minimize } f(x) \quad \text{subject to } x \in \Omega,\ z = g(x), \quad (P(z)), \ z \in U,
\]
where $f:\Omega \rightarrow \mathbb{R}$ and $g:\Omega \rightarrow Z$ are conventional point-to-point mappings, and $U \subset Z$ denotes an open neighborhood of the origin $0_Z$. We define the set-valued function
\[
V: U \rightrightarrows \mathbb{R}, \quad V(z) := \{f(x) : x \in \Omega, \, g(x) = z\}.
\]

The next theorem, which adapts Theorem~\ref{ThLagMult11_real} to the single-valued setting, formalizes in a nonconvex framework the separation-based principle underlying the Lagrange multiplier method. 
In the terminology of \cite[Section~8.4, Theorem 1]{Luenberger1969}, the supporting hyperplane is here replaced by a lower semicontinuous sublinear functional, whose existence guarantees exact penalization of the constraint without any additional assumptions.

\begin{theorem}  \label{teorema_LM_caso_escalar_univalorado}
Assume that $Y = \mathbb{R}$ and $Y_+ = \mathbb{R}_+$. Let
\[
r_0 := \inf \{ f(x) : x \in \Omega, \; g(x) = 0_Z \}.
\]
Assume that $r_0\in \mathbb{R}$ and that there exists a lower semicontinuous, sublinear mapping $\varphi : Z \to \mathbb{R}$ that satisfies 
\[
-\varphi(g(x)) < f(x)-r_0 
\quad \text{for all } x \in \Omega \text{ with } g(x) \neq 0_Z.
\]
Then  
\[
r_0 = \inf \{ f(x) + \varphi(g(x)) : x \in \Omega \}. \tag{$P[\varphi]$}
\]

Moreover, if $r_0 = f(x_0)$ for some feasible $x_0 \in \Omega$ (that is, $g(x_0)=0_Z$), 
then $x_0$ also minimizes $(P[\varphi])$, and
\[
\varphi(g(x_0)) = 0.
\]
\end{theorem}

The proof is omitted since it follows directly from that of the more general theorem.

\section{Verification of the General Assumptions under Lipschitz and Cone-Regularity Conditions}\label{sec:verification_lipschitz}

This section provides concrete conditions that ensure the existence of Lagrange multiplier processes. While Section~\ref{sec:general_framework} establishes the general theory under abstract geometric assumptions, here we identify verifiable hypotheses ---involving Lipschitz continuity and bounded bases--- that guarantee these assumptions hold. The main result, Theorem~\ref{teor:TML}, gives sufficient conditions for the nonemptiness of $\Gamma_{y_0}$, and we derive explicit penalty formulations in the scalar and single-valued cases. Specifically, it requires the composite set-valued mapping $V = F \circ G^{-1}$ to be Lipschitzian at the origin, ensuring uniform control of variations near the reference point; the ordering cone $Y_+$ to possess a bounded base, a property that is crucial for the separation arguments involved; and a nondegeneracy condition on $y_0$ relative to a $\delta$-perturbation of the cone, which guarantees that $y_0$ is suitably isolated within $V(0_Z)$. 

To this end, we start by introducing the terminology required for the formulation of Theorem~\ref{teor:TML}. In particular, we first recall the notion of cone perturbations (or dilations), which plays a central role in expressing the nondegeneracy condition on the reference point $y_0$. For any $y \in Y$ and $A \subset Y$, define the distance from $y$ to $A$ by
\[
d(y, A) := \inf\{ \|y - a\| : a \in A \}.
\]
Given a cone $C \subset Y$ and $0 < \varepsilon < 1$, set
\[
S_C := S_Y \cap C, 
\qquad 
S_{(C, \varepsilon)} := \{ y \in Y : d(y, C \cap S_Y) \le \varepsilon \}.
\]
Then
\[
S_{(C, \varepsilon)} = \operatorname{cl}((S_Y \cap C) + \varepsilon B_Y),
\]
so any $y \in S_{(C, \varepsilon)}$ can be represented as
\begin{equation} \label{puntos_S_C_epsilon_como_limite}
y = \lim_{n \to \infty} (y_n + z_n),
\end{equation}
for some sequences $y_n \in C \cap S_Y$ and $z_n \in \varepsilon B_Y$. Furthermore, given a bounded base \(B\) of \( C \), we set
\[
\sigma_{B} := \sup_{b \in B} \|b\| > 0,
\qquad
\delta_{B} := \inf_{b \in B} \|b\| > 0.
\]
\begin{definition}\label{def:epsilon_conic_neighbourhood} 
Let $Y$ be a normed space and $C \subset Y$ a cone. We define:
\begin{itemize}
\item[(i)] For $0 < \varepsilon < 1$, the cone 
\[
C_{\varepsilon} := \operatorname{cone}(S_{(C, \varepsilon)})
\]
is the $\varepsilon$-conic neighborhood of $C$.

\item[(ii)] If $C$ is convex and $B$ is a base of $C$, then for $0 < \varepsilon < \min\{1, \delta_B\}$ we define
\[
B_{\varepsilon} := \{ y \in Y : d(y, B) \le \varepsilon \}, 
\qquad 
C_{(B, \varepsilon)} := \operatorname{cone}(B_{\varepsilon}),
\]
called the Henig dilating cone associated with $C$ (via the base $B$).
\end{itemize}
\end{definition}

\begin{lema}\label{lem:base_acotada_dilation_Henig}
Let $X$ be a normed space and let $C \subset X$ be a cone. 
If $B$ is a bounded base of $C$, then there exists $\varepsilon'>0$ such that, for every $0<\varepsilon<\varepsilon'$, the Henig dilated cone $C_{(B,\varepsilon)}$ admits a bounded base.
\end{lema}

\begin{proof}
Since $B$ is a bounded base of $C$, by \cite[Theorem~1.47]{AliprantisTourky2007} there exist $f\in C^{\#}$ and $\varepsilon'>0$ such that 
\[
B=\{x\in C : f(x)=\varepsilon'\}.
\]
Fix $0<\varepsilon<\varepsilon'$, and set $B_{\varepsilon}:=\operatorname{cl}(B+\varepsilon B_X)$.  The set $B_{\varepsilon}$ is convex, bounded, and does not contain $0_X$.  
Indeed, if $x=\lim_n (y_n+\varepsilon b_n)$ with $y_n\in B$ and $b_n\in B_X$, then $\|y_n+\varepsilon b_n\|\geq \|y_n\|-\|\varepsilon b_n\|\geq \varepsilon'-\varepsilon>0$. Hence $\|x\|>0$

By Hahn--Banach, there exist $g\in X^*$ and $\gamma>0$ such that  
\[
g(0)=0<\gamma < \inf_{B_{\varepsilon}} g .
\]
Define
\[
B' := \{x\in C_{(B,\varepsilon)} : g(x)=\gamma\}.
\]

We first show that $B'$ is bounded.  
If $x\in B'$, then $x=\mu y$ for some $y\in B_{\varepsilon}$ and $\mu>0$, hence
\[
\gamma = g(x)=\mu g(y)\quad\Rightarrow\quad 
0<\mu < \frac{\gamma}{\gamma}=1.
\]
Since $B_{\varepsilon}$ is bounded, $B'$ is bounded as well.

Now we show that $B'$ is indeed a base for $C_{(B,\varepsilon)}$.  
By \cite[Theorem~1.47]{AliprantisTourky2007}, it is enough to check that $g(x)>0$ for all $x\in C_{(B,\varepsilon)}\setminus\{0\}$.  
If $x=\lambda y$ with $\lambda>0$ and $y\in B_{\varepsilon}$, then
\[
g(x)=\lambda g(y)>\lambda\gamma>0.
\]
Thus $B'$ is a bounded base of $C_{(B,\varepsilon)}$.
\end{proof}

We introduce the following terminology. For every $\varepsilon>0$, set
\[
C_{\varepsilon}^{\circ}:=\operatorname{cone}\bigl((S_Y \cap C)+\varepsilon B_Y\bigr).
\]
The next lemma summarizes the properties of these dilation cones that will be used later.

\medskip

\begin{lema} \label{lem:C_epsilon_mu_fdo}
Let \(Y\) be a normed space, let \(C \subset Y\) be a cone, and let \(\varepsilon, \mu \in (0,1)\). The following statements hold:
\begin{itemize}
    \item[(i)] If \(D\subset Y\) is another cone and \(C\subset D\), then \(C_{\varepsilon}^{\circ}\subset D_{\varepsilon}^{\circ}\).
    \item[(ii)]  \(\operatorname{cl}C_{\varepsilon}^{\circ} = C_{\varepsilon}\).
    \item[(iii)] \((C_{\varepsilon}^{\circ})_{\mu}^{\circ} \subset C_{\varepsilon + (1+\varepsilon)\mu}^{\circ} \subset C_{\varepsilon + 2\mu}^{\circ}\).
    \item[(iv)] If \(\varepsilon<\mu\), then \(C_{\varepsilon}\subset C_{\mu}^{\circ}\).
    \item[(v)] If \(\varepsilon<\mu<\tfrac13\), then  \((C_{\varepsilon})_{\mu} \subset C_{3\mu}\). 
    \item[(vi)] \((C_{\frac{\mu}{4}})_{\frac{\mu}{4}} \subset (C_{\frac{\mu}{4}})_{\frac{\mu}{3}} \subset C_{\mu}\).
\end{itemize}
\end{lema}

\begin{proof}
\noindent (i) Since \(C\cap S_Y\subset D\cap S_Y\), adding \(\varepsilon B_Y\) preserves inclusion, and taking conical hulls yields  
\[
C_{\varepsilon}^{\circ}=\operatorname{cone}(C\cap S_Y+\varepsilon B_Y)
    \subset 
\operatorname{cone}(D\cap S_Y+\varepsilon B_Y)=D_{\varepsilon}^{\circ}.
\]

\noindent (ii) Because \(0<\varepsilon<1\), \cite[Lemma 2.1.43(iii)]{GoeRiaTamZal2023} applies and gives  
\[
\operatorname{cl}C_{\varepsilon}^{\circ}
   =\operatorname{cone}\bigl(\operatorname{cl}((S_Y\cap C)+\varepsilon B_Y)\bigr)
   = C_{\varepsilon}.
\]

\noindent (iii) Let \(x\in (C_{\varepsilon}^{\circ})_{\mu}^{\circ}\). If \(x=0_Y\), the conclusion is immediate. Otherwise, write  
\[
x = t(u+\mu v),\qquad 
u = s(w+\varepsilon b),
\]
with \(t,s>0\), \(w\in C\cap S_Y\) and \(b,v\in B_Y\). Then  
\[
x = ts\Bigl(w + \varepsilon b + \frac{\mu}{s}v\Bigr).
\]
From \(\|u\|=1\) and \(u = s(w+\varepsilon b)\), we obtain  
\[
\frac1s = \|w+\varepsilon b\| \le 1+\varepsilon.
\]
Hence  
\[
x \in \operatorname{cone}\!\Bigl((C\cap S_Y) + (\varepsilon+(1+\varepsilon)\mu)B_Y\Bigr)
     = C_{\varepsilon + (1+\varepsilon)\mu}^{\circ}.
\]
The inclusion \(C_{\varepsilon + (1+\varepsilon)\mu}^{\circ}\subset C_{\varepsilon+2\mu}^{\circ}\) is immediate from \((1+\varepsilon)\mu\le 2\mu\).

\noindent (iv) Let \(x\in C_{\varepsilon}=\operatorname{cone}(S_{(C,\varepsilon)})\). If \(x=0_Y\), there is nothing to prove. Thus, we assume \(x\neq 0_Y\). Then there exist \(t>0\) and \(y\in S_{(C,\varepsilon)}\) such that \(x = ty\). Let \(\delta := d(y, C\cap S_Y)\). By the definition of \(S_{(C,\varepsilon)}\), we have \(\delta \leq \varepsilon < \mu\).

Choose a sequence \((z_n)\subset C\cap S_Y\) with \(\|y-z_n\|\to\delta\). Then, for some \(n_0\),  
\[
\|y-z_{n_0}\| < \mu,
\]
so \(y\in C\cap S_Y + \mu B_Y\) and therefore \(x\in C_{\mu}^{\circ}\).

\noindent (v) Since \(\varepsilon<\mu\), part (iv) yields \(C_{\varepsilon}\subset C_{\mu}^{\circ}\). Now, applying (i) and (iii),
\[
(C_{\varepsilon})_{\mu}^{\circ}
    \subset (C_{\mu}^{\circ})_{\mu}^{\circ}
    \subset C_{3\mu}^{\circ}.
\]
Taking closures and using (ii) gives  
\[
(C_{\varepsilon})_{\mu}
   = \operatorname{cl} \,(C_{\varepsilon})_{\mu}^{\circ}
   \subset \operatorname{cl}C_{3\mu}^{\circ}
   = C_{3\mu}.
\]
Let us note that, in the last equality, the assumption $3\mu<1$ is essential in order to apply (ii) correctly.

\noindent (vi) The first inclusion is trivial, and the second one is an immediate consequence of (v).
\end{proof}

In order to formulate the regularity assumptions required in Theorem~\ref{teor:TML}, we introduce a pointwise notion of Lipschitz continuity for set-valued mappings. This notion provides a convenient control of the values of a mapping with respect to a fixed reference point and will play a key role in the variational arguments developed below. Before introducing our notion of Lipschitzian behavior at a point, we recall the classical definition of Lipschitz continuity for set-valued mappings, which is well established in the literature.

\begin{definition}[{\cite[Definition~3.3.11(i)]{KTZ2015}}]\label{def:lipschitz_continuous}
Let $X$ and $Y$ be normed spaces, and let $U\subset X$ be a nonempty set.
A set-valued mapping $F:U\rightrightarrows Y$ is said to be Lipschitz continuous (on $U$)
if there exists a constant $L\ge 0$ such that
\[
F(x)\subset F(u)+L\|x-u\|\,B_Y,
\qquad \forall\, x,u\in U.
\]
The infimum of all such constants $L$ is called the exact Lipschitz bound of $F$ on $U$
and is denoted by $L_F\ge 0$.
\end{definition}
We now present a new and weaker, localized notion of Lipschitz behavior at a specific point, which will be useful in subsequent arguments.
\begin{definition}\label{def:lipschitz_at_point}
Let $X$ and $Y$ be normed spaces, and let $U\subset X$. A set-valued mapping $F:U\rightrightarrows Y$ is said to be Lipschitzian at $\bar x\in U$ if there exists a constant $L\ge 0$ such that
\[
F(w)\subset F(\bar x)+L\|w-\bar z\|\,B_Y,
\qquad \forall w\in U.
\]
The infimum of all such constants $L$ is called the exact Lipschitzian bound of $F$ at $\bar x$ and is denoted by $L_{F,\bar x}\ge 0$.
\end{definition}

The above notion is weaker than the notion of a set-valued mapping being Lipschitzian around $\bar x\in U$ as defined in \cite[Definition~3.3.11(ii)]{KTZ2015}. Indeed, in our definition the reference set $F(\bar x)$ is fixed, and the values $F(w)$, for $w\in U$, are viewed as uniformly controlled variations of $F(\bar x)$ according to a linear growth rate determined by the constant $L_{F,\bar x}$. This weaker notion is sufficient for our purposes, since the subsequent arguments only require uniform control of the values of $F$ with respect to a fixed reference point, typically $0_Z$. In particular, no comparison between $F(z_1)$ and $F(z_2)$ for arbitrary $z_1,z_2\in U$ is needed. This allows us to work under milder regularity assumptions while still obtaining the desired variational and optimality properties.

In the following example, we illustrate the relationship between the notion of local Lipschitz continuity introduced in this paper and the two previously studied Lipschitz-type notions from the literature.

\begin{example}
Let $X=Y=\mathbb{R}$, $U=\mathbb{R}$, and define the set-valued mappings $F,G:\mathbb{R}\rightrightarrows\mathbb{R}$ by
\[
F(x):=\{y\in\mathbb{R}: y\ge x^{2}\},\qquad
G(x):=\{y\in\mathbb{R}: y\ge x^{3}\},\qquad x\in\mathbb{R}.
\]
Then $F$ is Lipschitzian at $0$, but it is not Lipschitzian around $0$ in the sense of \cite[Definition~3.3.11(ii)]{KTZ2015}. Moreover, $G$ is not Lipschitzian at any point $x\in \mathbb{R}$.
\end{example}

The following lemma provides a sufficient condition ensuring that the composition of two set-valued mappings is Lipschitzian at a given point, together with an explicit estimate of the corresponding Lipschitz constant.

\begin{lema}\label{lem:composition_lipschitz_explicit}
Let $X,Y,Z$ be normed spaces, and let $U\subset Z$ and $V\subset X$. 
Assume that $F_{1}:U\rightrightarrows X$ is Lipschitzian at some $z\in U$, with $F_{1}(U)\subset V$, and that $F_{2}:V\rightrightarrows Y$ is Lipschitzian at every point $x\in F_{1}(z)$. 
Moreover, suppose that
\[
L_{F_{2},F_{1}(z)}:=\sup\{\,L_{F_{2},x}: x\in F_{1}(z)\,\}<+\infty.
\]
(The finiteness of $L_{F_{2},F_{1}(z)}$ is automatic, for instance, if $F_{1}$ is single-valued or if $F_{2}$ is Lipschitz continuous on $V$).

Then the composition $F_{2}\circ F_{1}$ is Lipschitzian at $z$, with Lipschitz constant explicitly given by
\begin{equation}\label{eq_comp_lips}
L_{F_{2}\circ F_{1},\,z}\leq L_{F_{2},F_{1}(z)}\,L_{F_{1},z}.    
\end{equation}

\end{lema}
\begin{proof}
Fix $z\in U$. We first show that there exists a constant $L\ge 0$ such that
\[
F_{2}(F_{1}(z'))\subset F_{2}(F_{1}(z)) + L\|z'-z\|\,B_Y,
\qquad \forall z'\in U.
\]

Let $z'\in U$ be arbitrary, and take $y'\in F_{2}(F_{1}(z'))$. Then there exists
$x'\in F_{1}(z')$ such that $y'\in F_{2}(x')$.
Fix any $L_{1}>L_{F_{1},z}$. By the Lipschitzian property of $F_{1}$ at $z$, we have
\[
F_{1}(z')\subset F_{1}(z)+L_{1}\|z'-z\|\,B_X,
\qquad \forall z'\in U,
\]
and hence there exists $x\in F_{1}(z)$ satisfying
\[
\|x'-x\|\le L_{1}\|z'-z\|.
\]

Now fix any $L_{2}>L_{F_{2},F_{1}(z)}$. Since $F_{2}$ is Lipschitzian at $x\in F_{1}(z)$
and $L_{2}>L_{F_{2},x}$, it follows that
\[
F_{2}(x')\subset F_{2}(x)+L_{2}\|x'-x\|\,B_Y,
\qquad \forall x'\in V.
\]
Consequently, there exists $y\in F_{2}(x)$ such that
\[
\|y'-y\|
\le L_{2}\|x'-x\|
\le L_{2}L_{1}\|z'-z\|.
\]

Since $L_{1}$ and $L_{2}$ depend only on $z$, we obtain
\[
F_{2}(F_{1}(z'))
\subset F_{2}(F_{1}(z))+L_{1}L_{2}\|z'-z\|\,B_Y,
\qquad \forall z'\in U,
\]
and thus $F_{2}\circ F_{1}$ is Lipschitzian at $z$.

Moreover, the previous inclusion implies
\[
L_{F_{2}\circ F_{1},\,z}
\le \inf\{\,L_{1}L_{2} : L_{1}>L_{F_{1},z},\; L_{2}>L_{F_{2},F_{1}(z)}\,\}.
\]
Since
\[
\inf\{\,L_{1} : L_{1}>L_{F_{1},z}\}=L_{F_{1},z}
\quad\text{and}\quad
\inf\{\,L_{2} : L_{2}>L_{F_{2},F_{1}(z)}\,\}=L_{F_{2},F_{1}(z)},
\]
we conclude that
\[
L_{F_{2}\circ F_{1},\,z}
\le L_{F_{2},F_{1}(z)}\,L_{F_{1},z}.
\]
This completes the proof.
\end{proof}

\begin{remark}
In general, the inequality \eqref{eq_comp_lips} cannot be replaced by an equality. 
Indeed, consider the functions $f,g:\mathbb{R}\to\mathbb{R}$ defined by
\[
f(x)=
\begin{cases}
1, & x=1,\\
0, & x\neq 1,
\end{cases}
\qquad
g(x)=
\begin{cases}
2, & x=2,\\
0, & x\neq 2.
\end{cases}
\]
Both functions are Lipschitzian at $0$, with minimal Lipschitz constant equal to $1$.
However, the composition $g\circ f$ is the constant zero function, and therefore its minimal Lipschitz constant at $0$ is equal to $0$.
\end{remark}

We now establish the main result of this section.

\begin{theorem} \label{teor:TML}
Let $y_0 \in \mathrm{ND}(P(0_Z))$ and assume:
\begin{itemize}
\item[(a)] The set-valued map \(V: U \rightrightarrows Y, \quad V(z) := F \circ G^{-1}(z)\) is Lipschitzian at $0_Z\in U$.
\item[(b)] The cone $Y_+$ has a bounded base.
\item[(c)] There exists $0 < \delta < 1$ such that 
\[
V(0) \cap [y_0 - (Y_+)_{\delta}] \subset \{ y_0 \}.
\]
\end{itemize}
Then $\Gamma_{y_0} \neq \varnothing$; that is, there exists $\Delta \in P(Z, Y)$ which is a Lagrange multiplier of $(P(0_Z))$ at $y_0$; which means that $y_0$ is a nondominated point of the program
\[
\text{Minimize } F(x) + \Delta(G(x)) \quad \text{subject to } x \in \Omega. \tag{$P[\Delta]$}
\]

Furthermore, if $y_0$ is a minimal point of $(P(0_Z))$ (i.e., if $y_0 \in F(x_0)$ for some feasible solution $x_0$), then $y_0$ is also a minimal point of $(P[\Delta])$ (achieved at $x_0$), and
\begin{equation} \label{eq:inclusion-henig_inicio}
\Delta(G(x_0)) \cap (-Y_+) = \{ 0_Y \}.
\end{equation}
\end{theorem}

The proof of Theorem \ref{teor:TML} follows from the nonemptiness of $\Gamma_{y_0}$ under assumptions (a)--(c). Indeed, by Theorem \ref{ThLagMult11}, the conclusion holds since \eqref{eq:inclusion-henig_inicio} follows from \eqref{InterseccionTeo_b'} and the fact that $Y_+$ is pointed (by assertion (b) in Theorem \ref{teor:TML}). Therefore, Theorem~\ref{teor:TML} will be proved once we establish the following result.

\begin{proposition} \label{prop:proposicion_01_sin_core}
Let $y_0 \in \mathrm{ND}(P(0_Z))$ and assume:
\begin{itemize}
\item[(a)] The set-valued map \(V: U \rightrightarrows Y, \quad V(z) := F \circ G^{-1}(z)\) is Lipschitzian at $0_Z\in U$.
\item[(b)] The cone $Y_+$ has a bounded base.
\item[(c)] There exists $0 < \delta < 1$ such that 
\[
V(0) \cap \big(y_0 - (Y_+)_{\delta}\big) \subset \{y_0\}.
\]
\end{itemize}
Then $\Gamma_{y_0} \neq \varnothing$.
\end{proposition}

We next present some technical results that will be instrumental in the proof of Proposition~\ref{prop:proposicion_01_sin_core}.

\begin{lema}  \label{lema:lema01}
Let $C \subset Y$ be a cone with a bounded base, let $A \subset Y$ be a nonempty set, and $L>0$. Suppose that there exist $\bar{y} \in Y$ and $0<\delta<1$ such that
\begin{equation}\label{condicion01}
[A-\bar{y}] \cap [-C_{\delta}] \subset \{0_Y\}.
\end{equation}
Then there exists a bounded base $B$ of $C$ such that
\begin{equation}\label{eq01}
d(-\lambda b, A-\bar{y}) \geq 2L\lambda,\qquad \text{for all } b\in B, \; \lambda>0.
\end{equation}
\end{lema}

\begin{proof} 
Pick  $B_1$ a bounded base of $C$. Then the set  $B_2:=\frac{1}{\delta_{B_1}}B_1$ is also a base of $C$, and  $\| b \| \geq 1$ for every $b \in B_2$.
Therefore,  for every $b \in B_2$ we have that 
\[
\frac{1}{\| b \|} B_Y(b, \delta) = B_Y\left( \frac{b}{\| b \|}, \frac{\delta}{\| b \|}\right)\subseteq  B_Y\left( \frac{b}{\| b \|},  \delta \right)\subseteq \left( S_Y\cap C\right) + \delta B_Y
\]
and consequently 
\begin{equation}\label{inclusion11}
B_Y(b, \delta) \subset C_{\delta} \setminus \{0_Y\}
\end{equation}
for every $b \in B_2$.
\[
\frac{2L}{\delta} \, B_Y(b, \delta) \subset   C_{\delta} \setminus \{0_Y\}
\]
and then,  for every $b \in B_2$ we get
\[
C_{\delta} \setminus \{0_Y\} \supset \frac{2L}{\delta} \, B_Y(b, \delta) = B_Y\left( \frac{2L}{\delta} \, b, \frac{2L}{\delta}  \delta\right)=
 B_Y\left( \frac{2L}{\delta} \, b, 2L \right)
\]
 Now,
since $(A- \bar{y})\cap [-C_{\delta}]\subset \{ 0_Y \}$ we obtain that 
\begin{equation} \label{ecuacion2}
B_Y\left( -\frac{2L}{\delta} \, b, 2L \right) \cap (A-\bar{y})=\varnothing.
\end{equation}
Finally, taking  $B:=\frac{2L}{\delta}  \, B_2$ we attain the base we are looking for, and from (\ref{ecuacion2}) we get that
\[
B_Y\left( - b, 2L\right) \cap (A-\bar{y})=\varnothing
\]
for any $b \in B$. Now, again because  $C_{\delta}$ is a cone,  for any $\lambda>0$ we have that 
\[
B_Y\left( \lambda b, 2L\lambda \right) \subset C_{\delta} \setminus  \{0_Y\} 
\]
and since
\[
  -C_{\delta}  \cap (A-\bar{y})\subset\{0_Y\}
\]
we get that
\[
B_Y\left( -\lambda  b, 2L\lambda \right) \cap (A-\bar{y})=\varnothing
\]
for every $b \in B $ and $\lambda >0$. 
We conclude the proof  noting that it is straightforward to check that the obtained base $B$ is bounded. 
\end{proof}

\begin{proposition} \label{Claim3}
Let \(C \subset Y\) be a convex cone with a bounded base, let \(H:Z \rightrightarrows Y\) be a set-valued map Lipschitzian at $0_Z$, and let \(\bar{y} \in Y\).  
Assume that there exists \(0<\delta<1\) such that
\begin{equation}\label{condicion01_bis}
[H(0_Z)-\bar{y}] \cap (-C_\delta) \subset \{0_Y\}.
\end{equation}
Then the following assertions hold:
\begin{itemize}
    \item[(i)] There exists a bounded base \(B\) of \(C\) such that
the set-valued map \(\Delta: Z \rightrightarrows Y\) defined by
    \begin{equation}\label{def_proceso}
    \operatorname{Gph}(\Delta)
    := \operatorname{cl}\, \operatorname{cone}(B_Z \times B),
    \end{equation}
    belongs to \(P(Z,Y)\) and 
    \begin{equation}\label{ecuacion6_bis_bis}
    \operatorname{Gph}(-\Delta)
    \cap \bigl(\operatorname{Gph}(H) - (0_Z,\bar{y})\bigr)
    \subset \{0_{Z\times Y}\}.
    \end{equation}
    \item[(ii)] The set-valued map \(\Delta\) also satisfies: 
    \begin{itemize}
        \item[(a)] \(C \subset \Delta(0_Z)\).
        \item[(b)] \(( -C ) \cap \Delta(0_Z)=\{0_Y\}\).
        \item[(c)] \(\|\Delta\|<\infty\). 
        \item[(d)] If $\operatorname{core}(C)\not = \varnothing$, then \(\operatorname{core}(\operatorname{Gph}(\Delta))\neq\varnothing\). 
    \end{itemize}
\end{itemize}
\end{proposition}

\begin{proof}
(i) To simplify notation, we define the set-valued mapping
\[
H_{\bar{y}}(z) := H(z) - \bar{y}, \quad \forall z \in Z.
\]
Hence,
\[
\operatorname{Gph}(H_{\bar{y}}) = \operatorname{Gph}(H) - (0,\bar{y}).
\]
There exists a constant $L>0$ such that $H$ is Lipschitzian at $0_Z$ with constant $L$. Then the mapping $H_{\bar y}$ is also Lipschitzian at $0_Z$ with the same constant $L$. Now, by \eqref{condicion01_bis}
and Lemma \ref{lema:lema01}, there exists a bounded base $B$ of $C$ such that
\begin{equation}\label{ecuacion3}
d(-\lambda b, H_{\bar{y}}(0_Z)) \geq 2L\lambda , \quad \forall b \in B,\, \forall \lambda>0.
\end{equation}
Now define the set-valued map \(\Delta : Z \rightrightarrows Y\) using formula \eqref{def_proceso}. 
By construction, \(\Delta\) is closed. Since \(B_Z \times B\) is convex (as both \(B_Z\) and \(B\) are convex), \(\Delta\) is convex as well. Moreover, \(\operatorname{dom}(\Delta)=Z\), because \(\pi_1(B_Z \times B)=B_Z\), where \(\pi_1\) denotes the projection onto \(Z\). To show that \(\Delta\) is pointed, observe first that \(0_{Z\times Y} \notin \operatorname{cl}(B_Z \times B)\) and that \(\operatorname{cl}(B_Z \times B)\) is bounded. Then, by \cite[Lemma~2.1.43]{GoeRiaTamZal2023},
\[
\operatorname{Gph}(\Delta)
= \operatorname{cl}\, \operatorname{cone}(B_Z \times B)
= \operatorname{cone}\bigl(\operatorname{cl}(B_Z \times B)\bigr).
\]
Hence, \(\operatorname{Gph}(\Delta)\) is a well-based cone (see \cite[Definition~2.1.42]{GoeRiaTamZal2023}), and by the comment following that definition, every well-based cone is pointed. Therefore, \(\operatorname{Gph}(\Delta)\) is pointed.

Now we verify \eqref{ecuacion6_bis_bis}.  
From
\begin{equation}\label{eq_menos_Delta}
\operatorname{Gph}(-\Delta)
   = \operatorname{cl}\,\operatorname{cone}(B_Z \times (-B)),
\end{equation}
it suffices to prove that
\begin{equation}\label{ecuacion6}
\Bigl( \operatorname{cl}\,\operatorname{cone}(B_Z \times (-B))
       \setminus \{0_{Z\times Y}\}\Bigr)
\;\cap\;\bigl( \operatorname{Gph}(H)-(0_Z,\bar y) \bigr)=  \varnothing.   
\end{equation}
Take first \((\hat z,\hat y)\in \operatorname{cone}(B_Z \times (-B))\setminus\{0_{Z\times Y}\}\).  
Then \((\hat z,\hat y)=\lambda (z',-b')\) for some \((z',-b')\in B_Z\times (-B)\) and \(\lambda>0\).  
Using \eqref{ecuacion3} and \(\|z'\|\le 1\),
\[
d(\hat y,H_{\bar y}(0_Z))
= d(-\lambda b',H_{\bar y}(0_Z))
\ge 2L\lambda
\ge 2L\lambda\|z'\|
= 2L\|\hat z\|.
\]
Hence,
\begin{equation}\label{ecuacion4}
\hat y \notin H_{\bar y}(0_Z) + \tfrac{3}{2}L\|\hat z\|\,B_Y .
\end{equation}
By continuity of the norm, \eqref{ecuacion4} also holds for all  
\((\hat z,\hat y)\in \operatorname{cl}\operatorname{cone}(B_Z \times (-B))\setminus\{0_{Z\times Y}\}\).

Now take \((\hat z,\hat y)\in \operatorname{Gph}(H)-(0_Z,\bar y)=\operatorname{Gph}(H_{\bar y})\).  
Since \(H_{\bar y}\) is Lipschitzian at $0_Z$ with constant \(L\),
\[
\hat y\in H_{\bar y}(\hat z)
   \subset H_{\bar y}(0_Z) + L\|\hat z\|\, B_Y
   \subset H_{\bar y}(0_Z) + \tfrac{3}{2}L\|\hat z\|\, B_Y .
\]
This contradicts \eqref{ecuacion4}, proving \eqref{ecuacion6}.  

\noindent (ii) (a) Since
\begin{equation}\label{eq:prop_nueva_CcontenidoDelta}
\{0_Z\}\times B \subset \operatorname{cl}\,\operatorname{cone}(B_Z \times B) = \operatorname{Gph}(\Delta),
\end{equation}
we obtain
\[
\operatorname{cone}(\{0_Z\}\times B)
= \{0_Z\}\times \operatorname{cone}(B)
\subset \operatorname{Gph}(\Delta).
\]
Hence,
\[
C = \operatorname{cone}(B) \subset \Delta(0_Z).
\]

\noindent (b) By assertion (i), the cone \(\operatorname{Gph}(\Delta)\subset Z\times Y\) is pointed.  
In particular, $\Delta(0_Z) \subset Y$ is also pointed, and this implies
\[
(-\Delta(0_Z))\cap \Delta(0_Z)=\{0_Y\}.
\]
Therefore,
\[
(-C)\cap \Delta(0_Z)\subset (-\Delta(0_Z))\cap \Delta(0_Z)=\{0_Y\}.
\]

\noindent (c) By \eqref{eq:prop_nueva_CcontenidoDelta}, for every \(z\in S_Z\) we have \(B\subset \Delta(z)\).  
Then, using \eqref{def:norma_proceso}, 
\[
\|\Delta\|
= \sup_{z\in S_Z} d(0_Y,\Delta(z))
\leq \inf_{b\in B} \|b\|
<\infty.
\]

\noindent (d) Now, let us check that $\operatorname{core}(\operatorname{Gph}(\Delta)) \neq \varnothing$. Fix some $\bar{y} \in \operatorname{core}(C)$. Since $B$ is a base for $C$, there exists $\bar{y}_b \in B$ and $\lambda_{\bar{y}} > 0$ such that $\bar{y} = \lambda_{\bar{y}} \bar{y}_b$. Clearly, $\bar{y}_b \in \operatorname{core}(C)$. We aim to show that $(0_Z, \bar{y}_b) \in \operatorname{core}(\operatorname{Gph}(\Delta))$. Recall from \cite[Section 6.3]{KTZ2015} that every real linear space $X$ can be endowed with the core convex topology $\tau_c(X)$, the strongest locally convex topology satisfying $\operatorname{int}_{\tau_c}(A) = \operatorname{core}(A)$ for every convex set $A \subset X$. Thus, it suffices to find $\mathcal{V}\in \tau_c(Z\times Y)$ such that $(0_Z,\bar{y}_b)\in \mathcal{V}\subset \operatorname{cone}(B_X\times B)$. By \cite[Proposition 6.3.1 (iii)]{KTZ2015}, pick a convex set $\mathcal{U}\in \tau_c(Y)$ with $\bar{y}_b \in \mathcal{U} \subset C$. Fix $0 < \alpha < \frac{\delta_B}{2}$. Since $\operatorname{int} B_Y(\bar{y}_b, \alpha) = \operatorname{core}(B_Y(\bar{y}_b, \alpha)) \in \tau_c(Y)$ by \cite[Proposition 2.3.2 (iii)]{KTZ2015}, define
\[
\mathcal{U}' := \mathcal{U} \cap \operatorname{core}(B_Y(\bar{y}_b, \alpha)).
\]
Then $\bar{y}_b \in \mathcal{U}' \subset C$, $\mathcal{U}' \in \tau_c$, and $\mathcal{U}'$ is convex. For every $y \in \mathcal{U}'$, we have $\| y \| > \frac{\delta_B}{2}$ since
\[
\| y \| = \| \bar{y}_b -(y-\bar{y}_b) \| \geq \| \bar{y}_b \| -\|y-\bar{y}_b \| > \delta_B-\frac{\delta_B}{2}=\frac{\delta_B}{2},
\]
where the last inequality uses $\|\bar{y}_b\|\geq \delta_B$ and $\|y-\bar{y}_b \| < \alpha < \frac{\delta_B}{2}$. For each $y \in \mathcal{U}'$, write $y= \lambda_{y} y_{b}$ with $y_{b} \in B$ and $\lambda_y>0$. Then $\|y\|= \lambda_{y} \|y_{b}\|$, and using $\frac{\delta_B}{2} < \|y\|\leq 2 \sigma_B$ and $\delta_B \leq \|y_b\|\leq \sigma_B$, we obtain $\lambda_{y} > \frac{\delta_B}{2 \sigma_B}$. Set $M:=\min \left\{ 1, \frac{\delta_B}{2\sigma_B } \right\}>0$. By \cite[Proposition 2.3.2 (iii), Proposition 6.3.2 (ii)]{KTZ2015},
\[
(0_Z, \bar{y}_b )\in \operatorname{core}(B_Z(0_Z,M)) \times \mathcal{U}'= \operatorname{core}(B_Z(0_Z,M)\times \mathcal{U}')\in \tau_c(Z \times Y).
\]
To complete the proof, we show $\operatorname{core}(B_Z(0_Z,M)) \times \mathcal{U}' \subset \operatorname{cone}(B_Z \times B)$. Fix $(z,y) \in \operatorname{core}(B_Z(0_Z,M)) \times \mathcal{U}'$. Write $y= \lambda_{y} y_{b}$ with $y_b \in B$ and $\lambda_y > \frac{\delta_B}{2\sigma_B}$. Then $(z,y)=\lambda_y(\frac{z}{\lambda_y},y_b)$. Since $\|z\|<M \leq \frac{\delta_B}{2\sigma_B}<\lambda_y$, we have $\left\|\frac{z}{\lambda_y}\right\|<1$, so $(\frac{z}{\lambda_y},y_b)\in B_Z\times B$, completing the proof.
\end{proof}

\begin{proof}[Proof of Proposition \ref{prop:proposicion_01_sin_core}]
Set $\delta':=\delta/4$ and choose $0<\rho<\delta'$. Since $Y_+$ has a bounded base, by \cite[Lemma~3.11]{GarciaCastanoMelguizo2024} there exists a bounded base $B$ of $Y_+$ such that 
\[
(Y_+)_{(B,\nu)} \subset (Y_+)_{(B,\rho)} \subset (Y_+)_{\rho}
\quad\text{for all } 0<\nu\le\rho.
\]
Let $\mu:=\rho/2$ and set $Y'_+:=(Y_+)_{(B,\mu)}$. By Lemma~\ref{lem:base_acotada_dilation_Henig}, we may assume that $Y'_+$ has a bounded base. Then $Y'_+$ is a convex cone with bounded base and
\[
Y_+ \subset Y'_+ \subset (Y_+)_{\rho}, 
\qquad
\operatorname{int}(Y'_+)=\operatorname{core}(Y'_+) \neq \varnothing.
\]
Using these inclusions and Lemma~\ref{lem:C_epsilon_mu_fdo} (vi), we obtain
\[
Y_+ \subset Y'_+ \subset (Y_+')_{\delta'} 
   \subset ((Y_+)_{\rho})_{\delta'}
   \subset ((Y_+)_{\frac{\delta}{4}})_{\frac{\delta}{4}}
   \subset (Y_+)_{\delta}.
\]
By the assumption of the proposition,
\[
[V(0)-y_0] \cap (-(Y_+)_{\delta}) = \{0_Y\}.
\]
Hence, using the above inclusions,
\[
[V(0)-y_0] \cap (-(Y'_+)_{\delta'}) = \{0_Y\}.
\]
Applying Proposition \ref{Claim3} to $Y'_+$, $V=F\circ G^{-1}$, $y_0$ and $\delta'$, we obtain $\Delta\in P(Z,Y)$ such that  
\[
\operatorname{core}(\operatorname{Gph}(\Delta))\neq\varnothing,
\qquad
\operatorname{Gph}(-\Delta)\cap [\operatorname{Gph}(V)-(0,y_0)]
   =\{0_{Z\times Y}\},
\]
\[
Y'_+ \subset \Delta(0_Z),
\qquad
(-Y'_+)\cap \Delta(0_Z)=\{0_Y\}.
\]
Finally, since $Y_+\subset Y'_+$, also
\[
Y_+ \subset \Delta(0_Z),
\qquad
(-Y_+)\cap \Delta(0_Z)=\{0_Y\},
\]
so $\Delta\in\Gamma_{y_0}$ and therefore $\Gamma_{y_0}\neq\varnothing$.
\end{proof}

As a consequence of the arguments employed in the proofs just established for Propositions \ref{prop:proposicion_01_sin_core} and \ref{Claim3}, we can formulate the following remark, which will be used in the proof of a subsequent result.
\begin{remark}\label{remark:conclusion_Delta_con_definicion}
Let $C\subset Y$ be a convex cone with a bounded base $B$, and define $\Delta:Z\rightrightarrows Y$ by \eqref{def_proceso}. Then $\Delta\in P(Z,Y)$, and assertions \textnormal{(a)}--\textnormal{(d)} in part~\textnormal{(ii)} of Proposition~\ref{Claim3} hold. Moreover, if $\Delta$ also satisfies condition \eqref{ecuacion6_bis_bis}, then the arguments used in the proof of Proposition~\ref{prop:proposicion_01_sin_core} can be readily adapted to show that $\Delta\in \Gamma_{y_0}$.
\end{remark}

\begin{remark}
By the Closed Graph Theorem \cite[Theorem~2.2.6]{aubin1990set} and the Open Mapping Theorem \cite[Theorem~2.2.1]{aubin1990set}, if $F:X\rightrightarrows Y$ and $G:X\rightrightarrows Z$ are closed convex processes between Banach spaces such that $\operatorname{Dom}(F)=X$ and $\operatorname{Im}(G)=Z$, then both $F$ and $G^{-1}$ are Lipschitz continuous. Consequently, their composition $V:=F\circ G^{-1}$ is also Lipschitz continuous. In particular, $V$ is Lipschitzian at $0_Z$.
\end{remark}

Since closed convex processes in set-valued analysis play a role analogous to that of linear operators in classical analysis, optimization programs involving such processes $F$ and $G$ can be regarded as the set-valued counterparts of linear programs. In light of the previous remark, the next theorem fits naturally into this framework under mild assumptions.

\medskip

We again consider the particularly relevant case $Y=\mathbb{R}$ and $Y_+=\mathbb{R}_+$.



\begin{theorem} \label{teor:TML_real}
Assume that $Y = \mathbb{R}$ and $Y_+ = \mathbb{R}_+$. Let
\[
r_0 := \inf \{ F(x) : x \in \Omega, \;  0_Z \in G(x) \},
\]
and assume that $r_0\in \mathbb{R}$. If the set-valued mapping $V=F\circ G^{-1}$ is Lipschitzian at $0_Z$ —for example, when $F$ and $G$ are convex closed processes, $\operatorname{Dom}(F)=X$, $\operatorname{Im}(G)=Z$, and $X$ and $Z$ are Banach spaces— then $\mathcal{S}_{r_0}\neq\varnothing$. As a consequence, every $\varphi\in \mathcal{S}_{r_0}$ is a Lagrange multiplier of $(P(0_Z))$ at $r_0$; that is,  
$r_0$ is the infimum of the program
\begin{equation*}\label{DualPr_r_lipschitz}
\text{Minimize } F(x) + \varphi(G(x)) \quad \text{subject to } x \in \Omega. \tag{$P[\varphi]$}
\end{equation*}
Furthermore, if $r_0$ is a minimum of $(P(0_Z))$ (that is, if $r_0 \in F(x_0)$ for some feasible solution $x_0$), then $r_0$ is also a minimal point of $(P[\varphi])$ (achieved at $x_0$), and
\begin{equation}\label{InterseccionTeo_r}
\varphi(z) \geq  0  \quad \text{for every } z \in G(x_0).
\end{equation}
\end{theorem}

\begin{proof}
It is clear that assumptions (a) and (b) of Theorem~\ref{teor:TML} are satisfied, so we only need to verify (c). Since $Y_+=\mathbb{R}_+$, we have $(Y_+)_{\delta}=\mathbb{R}_+$ for every $0<\delta<1$. As $r_0=\inf V(0_Z)$, it follows that 
\[
V(0_Z)\cap [\,r_0-(Y_+)_{\delta}\,]
 = V(0_Z)\cap (-\infty,r_0]
 \subset \{r_0\}.
\]
Therefore, by Theorem~\ref{teor:TML}, we obtain $\Gamma_{r_0}\neq\varnothing$.  
Then, by Proposition~\ref{proposition:biyeccion_Gamma_Gamma'}, we deduce that $\mathcal{S}_{r_0}\neq\varnothing$.  
Finally, by Theorem~\ref{ThLagMult11_real}, every $\varphi\in\mathcal{S}_{r_0}$ satisfies the conclusion of the theorem.
\end{proof}

Next, we continue working in the framework $Y = \mathbb{R}$ and $Y_{+} = \mathbb{R}_{+}$, now focusing on the particular case in which $F$ and $G$ are single-valued mappings rather than set-valued ones. In this setting, we denote them by $f$ and $g$, respectively. Before establishing the corresponding result, we state and prove a technical lemma.

\begin{lema}\label{lem: 3 conos}
Let $Z$ be a normed space, $B_Z=\{z\in Z:\|z\|\le 1\}$, and $\alpha>0$. Then,
\[
\operatorname{cone}(B_Z\times \{\alpha\})=\{(x,r)\in Z\times\mathbb{R} : r\ge\alpha\|x\|\}.
\]
\end{lema}

\begin{proof}
By definition
\[
\operatorname{cone}(B_Z\times \{\alpha\})=\{\lambda (x,\alpha) : \lambda\ge 0,\ x\in B_Z\}.
\] 
\noindent If $(x,r)=\lambda(z,\alpha)$ with $z\in B_Z$, then $\|z\|\le1$ and hence
\[
\|x\|=\lambda\|z\|\le\lambda=r/\alpha.
\]
Thus $r\ge\alpha\|x\|$. Conversely, given $(x,r)$ with $r\ge\alpha\|x\|$, set $\lambda=r/\alpha$ and, if $x\neq0$, let $z:=x/\lambda$; then $\|z\|\le1$ and $(x,r)=\lambda(z,\alpha)$.  
For $x=0$, take $z=0$.
The proof is complete.

\end{proof}

\begin{theorem}\label{teor:TML_real_single_valued}
Let $Y=\mathbb{R}$, $Y_+=\mathbb{R}_+$, and set
\[
r_0 := \inf\{f(x) : x\in\Omega,\ g(x)=0_Z\},
\]
assuming $r_0\in\mathbb{R}$.  
Suppose that the (possibly set-valued) map $V=f\circ g^{-1}$ is Lipschitzian at $0_Z$. Then, for any $\mu>L_{V,0_Z}$, we have
\[
r_0 = \inf\{f(x)+\mu \, \|g(x)\| : x\in\Omega\}. \tag{P[$\mu$]}
\]
Moreover, if $r_0 = f(x_0)$ for some feasible $x_0 \in \Omega$, then $x_0$ also solves $(P[\mu])$.
\end{theorem}

\begin{proof}
Since $Y_+=\mathbb{R}_+$, every $\mu>0$ is a base for it. Fix some $\mu>0$ and define $\Delta_{\mu}:Z\rightrightarrows \mathbb{R}$ by
\[
\operatorname{Gph}(\Delta_{\mu}):=\operatorname{cone}(B_Z\times\{\mu\}).
\]
By Remark~\ref{remark:conclusion_Delta_con_definicion}, $\Delta_{\mu}\in P(Z,\mathbb{R})$, and assertions (a)--(d) in part~\textnormal{(ii)} of Proposition~\ref{Claim3} hold. 

Let us now check that $\Delta_{\mu}$ also satisfies condition \eqref{ecuacion6_bis_bis} for $\mu$ large enough. In particular, we will show that 
\begin{equation}\label{ineq:aux-inclusion}
\operatorname{Gph}(-\Delta_{\mu})\cap
\bigl(\operatorname{Gph}(f\circ g^{-1})-(0_Z,r_0)\bigr)
\subset\{(0_Z,0)\},
\end{equation}
for every $\mu>L_{V,0_Z}$. Indeed, let $\mu>L_{V,0_Z}$. Then there exists $\gamma$ such that
$L_{V,0_Z}\le \gamma<\mu$ and, for every $(z,r)\in \operatorname{Gph}(f\circ g^{-1})$,
there exist $s_{0}\in f\circ g^{-1}(0_{Z})$ and $b\in[-1,1]$ satisfying
\[
r=s_{0}+\gamma\|z\|\, b \ge r_{0}-\gamma\|z\|.
\]
Since $\mu>\gamma$, it follows that $r-r_{0}>-\mu\|z\|$.

By Proposition~\ref{lem: 3 conos}, 
\[
\operatorname{Gph}(-\Delta_{\mu})=\{(z,-r)\in Z\times\mathbb{R} : (z,r)\in \operatorname{Gph}(\Delta_{\mu})\}=\{(z,r)\in Z\times\mathbb{R} : r\le -\mu\|z\|\}.
\]
Therefore, for every $z\neq 0_{Z}$ and $r\in f\circ g^{-1}(z)$, we have $(z,\, r-r_{0})\notin \operatorname{Gph}(-\Delta_{\mu})$, and \eqref{ineq:aux-inclusion} follows. As a consequence, $\Delta_{\mu} \in \Gamma_{r_0}$. Now, apply Proposition~\ref{proposition:biyeccion_Gamma_Gamma'}, in particular the definition of $\Upsilon$ given by \eqref{eq_defi_Upsilon}. Define $S_{\mu}:=\Upsilon^{-1}(\Delta_{\mu})\in \mathcal{S}_{r_0}$. Then $S_{\mu}(x)=\mu \|x\|$ for every $x\in Z$. Thus, the sublinear map $S_{\mu}$ belongs to $\mathcal{S}_{r_0}$ for every $\mu>L_{V,0_Z}$.

Finally, Theorem~\ref{teorema_LM_caso_escalar_univalorado} implies
\[
r_0=\inf\{f(x)+S_{\mu}(g(x)):x\in\Omega\}=\inf\{f(x)+\mu\|g(x)\|:x\in\Omega\},
\]
for every $\mu>L_{V,0_Z}$. If $x_0\in\Omega$ satisfies $g(x_0)=0_Z$ and $f(x_0)=r_0$, then
$f(x_0)+\mu\|g(x_0)\|=r_0$, 
so $x_0$ minimizes $(\mathrm{P}[\mu])$.
\end{proof}

\begin{coro}\label{coro:TML_real_single_valued}
Let $Y=\mathbb{R}$, $Y_{+}=\mathbb{R}_{+}$, and define
\[
r_0 := \inf\{f(x) : x\in\Omega,\ g(x)=0_Z\},
\]
assuming that $r_0\in\mathbb{R}$.  
Suppose that one of the following conditions holds:
\begin{itemize}
\item[(a)] The mappings $f$ and $g^{-1}$ are Lipschitz continuous, with constants $L_f$ and $L_{g^{-1}}$, respectively,
and $\mu>L_fL_{g^{-1}}$.
\item[(b)] The mapping $g^{-1}$ is Lipschitzian at $0_Z$, $f$ is Lipschitzian at every point
$x\in\Omega\cap g^{-1}(0_Z)$, and, in addition, one of the following conditions holds:
\begin{itemize}
\item[(b1)] $L_{f,g^{-1}(0_Z)}:=\sup\{L_{f,x}\colon x \in g^{-1}(0)\}<\infty$ and $\mu> L_{f,g^{-1}(0_Z)}\,L_{g^{-1},0_Z}$; 
\item[(b2)] $f$ is Lipschitz continuous with constant $L_f$, and $\mu>L_f\,L_{g^{-1},0_Z}$.
\end{itemize}
\end{itemize}
Then
\[
r_0=\inf\{f(x)+\mu\|g(x)\|:x\in\Omega\}. \tag{$\mathrm{P}[\mu]$}
\]
Moreover, if $r_0=f(x_0)$ for some feasible point $x_0\in\Omega$, then $x_0$ is also a
solution of $(\mathrm{P}[\mu])$.
\end{coro}

\begin{proof}
In case \textnormal{(a)}, the composition $f\circ g^{-1}$ is Lipschitz continuous, with Lipschitz constant $L_fL_{g^{-1}}$. In particular,
$f\circ g^{-1}$ is Lipschitzian at $0_Z$ and satisfies
\[
L_{f\circ g^{-1},\,0_Z}\le L_fL_{g^{-1}}.
\]
Therefore, Theorem~\ref{teor:TML_real_single_valued} applies and yields the desired conclusion.

In case \textnormal{(b)}, Lemma~\ref{lem:composition_lipschitz_explicit} ensures that
$f\circ g^{-1}$ is Lipschitzian at $0_Z$ with a constant strictly smaller than $\mu$.
Hence, Theorem~\ref{teor:TML_real_single_valued} applies again and the result follows.
\end{proof}

\section{An Application to Set-Valued Vector Equilibrium Problems}\label{sec:equilibrium_application}
 
Having established the abstract Lagrange multiplier theory and its verification under Lipschitz conditions, we now demonstrate its applicability to vector equilibrium problems.

\noindent Vector equilibrium problems provide a unifying framework for a wide class of models arising in vector optimization, game theory, and variational analysis.
Their defining feature is the comparison of feasible deviations with respect to
a partial order induced by a convex cone in the outcome space.
Rather than focusing on scalarization techniques, we adopt a genuinely
set-valued and vectorial perspective, in which equilibrium conditions are
reformulated as constrained optimization problems in the image space.

The central idea in our approach is to embed the equilibrium condition into a parametric
optimization framework and to interpret the resulting multiplier as a geometric
object encoding admissible directions of comparison and constraint interaction.

\medskip

\noindent
\textbf{Set-valued vector equilibrium formulation.}

Let $X$ and $Y$ be normed spaces, and let $Y_+ \subset Y$ be a closed convex cone
inducing a partial order on $Y$.
Consider a set-valued bifunction
\[
\bar{F} : X \times X \rightrightarrows Y,
\]
which is single-valued on the diagonal of $X \times X$, that is, $\bar{F}(x,x)$ is a
singleton for every $x \in X$.
Let $\Sigma \subset X$ be a nonempty feasible set satisfying the consistency
condition
\[
\bar{F}(x,x) = \{0_Y\}
\quad \text{for all } x \in \Sigma,
\]
which models the absence of deviation effects at a given feasible point.

The associated \emph{set-valued vector equilibrium problem} $(E)$ consists in
finding a point $x_0 \in \Sigma$ such that no admissible deviation yields a strict
improvement with respect to the order induced by $Y_+$.

\begin{definition}
A point $x_0 \in \Sigma$ is said to be a solution of $(E)$ if $0_Y$ is a minimal element of the set
\[
\bar{F}(x_0,\Sigma) := \bigcup_{x \in \Sigma} \bar{F}(x_0,x),
\]
that is, if
\[
\bar{F}(x_0,x) \cap \bigl(-Y_+ \bigr) \subset Y_+,
\quad \forall x \in \Sigma.
\]
\end{definition}

This condition expresses a stability property: at an equilibrium point $x_0$,
every feasible deviation produces outcomes that are either dominated by, or
incomparable with, the reference value $0_Y$ with respect to the ordering cone
$Y_+$.

\medskip
\noindent
\textbf{Reduction to a constrained optimization problem.}
To connect the equilibrium problem $(E)$ our the Lagrange multiplier framework,
we fix a candidate equilibrium point $x_0 \in \Sigma$ and define the mappings
\[
F : X \rightrightarrows Y, \qquad  F(x) := \bar{F}(x_0,x),
\]
and
\[
G : X \rightrightarrows Y, \qquad G(x) := \bar{F}(x,x).
\]

The consistency condition implies that $G(x_0)=\{0_Y\}$, and hence $x_0$ is feasible for the constrained set-valued optimization problem
\[
\text{Minimize } F(x)
\quad \text{subject to } x \in X,\; 0_Y \in G(x),
\]
which is exactly of the form $(P(0_Y))$ studied in the previous sections with $Z=Y$. 
Moreover, since $x_0$ is a solution of (E), it follows that $0_Y \in \mathrm{ND}(P(0_Y))$, and therefore the hypothesis of Theorem~\ref{ThLagMult11} reduces, in this setting, to the condition $\Gamma_{0_Y} \neq \varnothing$.

In this reformulation, the equilibrium requirement is encoded as a constraint,
while the deviation map $F$ plays the role of the objective function.

\medskip
\noindent
\textbf{Existence and interpretation of a multiplier.}
The following result is a direct consequence of the general Lagrange multiplier  theory.

\begin{theorem}\label{Teorema_Equilibrio_E}
Let $x_0 \in \Sigma$ be a solution of the equilibrium problem $(E)$, and define
$F$ and $G$ as above.
Assume that the hypotheses of Theorem~\ref{ThLagMult11} are satisfied.
Then there exists a convex process $\Delta \in P(Y,Y)$ such that $0_Y$ is a
minimal point of the program
\begin{equation*}
\text{Minimize } F(x) + \Delta( G(x))
\quad \text{subject to } x \in \Omega. \tag{$P[\Delta_E]$}
\end{equation*}
Moreover,
\begin{equation}\label{InterseccionTeo_b'_E}
\Delta(G(x_0)) \cap (-Y_+) \subseteq Y_+.
\end{equation}
\end{theorem}

The following example illustrates Theorem~\ref{Teorema_Equilibrio_E} in an infinite-dimensional setting with a nonstandard ordering cone.

\begin{example}\label{ex:equilibrium_slanted_cone}
Let $X := C[0,1]$ with the supremum norm and $Y := \mathbb{R}^2$ with the Euclidean norm. Define the ordering cone
\[
Y_+ := \Bigl\{(y_1,y_2)\in\mathbb{R}^2 : \tfrac{\sqrt{2}}{2}\,y_1 \le y_2 \le \tfrac{\sqrt{3}}{2}\,y_1, \ y_1 \ge 0\Bigr\}\subset \mathbb{R}^2,
\]
which is closed and convex. Let $\Sigma := \{u\in X : u(0)=0\}$ and define $\bar F:\Omega\times\Omega \to Y$ by
\[
\bar F(u,v) := \Bigl(u(0)^2 + (v(0)-u(0))^2, \int_0^1 (u(s)-v(s))\,ds\Bigr).
\]

\emph{Equilibrium point.} For $u\in\Sigma$, we have $\bar F(u,u)=(0,0)$, verifying the consistency condition. Setting $u_0 := 0 \in \Sigma$, for any $v\in\Sigma$ we obtain $\bar F(u_0,v)=(v(0)^2, -\int_0^1 v(s)\,ds)\in\mathbb{R}_+ \times \mathbb{R}$, which implies $\bar F(u_0,\Sigma)\cap(-Y_+\setminus\{0_Y\})=\varnothing$. Hence $u_0$ is a solution of $(E)$.

\emph{Multiplier process.} Define $F(v) := \bar F(u_0,v)$ and $G(v) := \bar F(v,v)$. Consider the process $\Delta:Y\rightrightarrows Y$ given by
\[
\Delta(z) := \Bigl\{y\in\mathbb{R}^2 : y_2 - \tfrac{\sqrt{2}}{2}y_1 \ge \|z\|, \ \tfrac{\sqrt{3}}{2}y_1 - y_2 \ge \|z\|, \ y_1 \ge 0\Bigr\}.
\]
One verifies directly that $\operatorname{Gph}(\Delta)$ is a closed convex cone by noting that the defining inequalities are preserved under positive scaling and convex combinations, and that the map $(z,y)\mapsto (y_2 - \tfrac{\sqrt{2}}{2}y_1 - \|z\|, \tfrac{\sqrt{3}}{2}y_1 - y_2 - \|z\|)$ is continuous with $\operatorname{Gph}(\Delta)$ being its inverse image of $[0,\infty)\times [0,+\infty)$.

\emph{Penalized problem.} For any $v\in\Omega$, we have $F(v)+\Delta(G(v)) \subset \mathbb{R}_+ \times \mathbb{R}$. Since $(\mathbb{R}_+ \times \mathbb{R}) \cap (-Y_+)=\{0_Y\}$, it follows that $0_Y$ is a minimal point of the penalized problem $(P[\Delta_E])$. Moreover, $\Delta(G(u_0)) \cap (-Y_+) = \{0_Y\}$, verifying condition~\eqref{InterseccionTeo_b'_E}.
\end{example}

\section*{Funding}
The authors have been supported by Project PID2021-122126NB-C32, funded by \\ MICIU/AEI/10.13039/501100011033 and by FEDER, A way of making Europe.

\end{document}